\documentclass{amsart}
\usepackage{amssymb,amsthm,amsmath}

\theoremstyle{plain}
\newtheorem{theorem}{Theorem}
\newtheorem{corollary}{Corollary}

\newtheorem{proposition}{Proposition} 

\theoremstyle{definition}
\newtheorem{definition}{Definition}

\theoremstyle{remark}

\numberwithin{equation}{section}

\newcommand\R{\mathbb{R}}
\newcommand\RP{\mathbb{RP}}

\newcommand\gl{{\rm gl}}

\def\gl{\mathfrak{gl}}

\def\r{{\bf r}}
\def\c{{\bf c}}

\def\h{{\bf h}}

\def\kb{{\bf k}}

\renewcommand{\k}[1]{\kappa_{#1}}

\def\SL{\mathrm{SL}}
\def\GL{\mathrm{GL}}

\def\O{\mathrm{O}}
\def\SO{\mathrm{SO}}
\def\F{\mathcal{F}}
\def\D{\mathcal{D}}
\def\P{\mathcal{P}}

\newcommand\p{{\mathfrak p}}
\newcommand\g{{\mathfrak g}}

\newcommand\n{{\mathfrak n}}

\newcommand\so{{\mathfrak so}}

\newcommand\K{{\mathcal K}}

\newcommand\Lo{{\mathcal L}}
\newcommand\Hop{{\mathcal H}}

\newcommand\0{\mathbf{0}}

\newcommand\Fop{{\mathcal F}}

\newcommand\var[1]{\frac{\delta #1}{\delta L}}

\newcommand\rr{{\bf r}}

\newcommand\norm[1]{|\!|#1|\!|}

\begin{document}

\title[Poisson brackets on Grassmannians]
      {Geometric Poisson brackets on Grassmannians and conformal spheres}
\author{G. Mar\'\i~Beffa}
\address{Mathematics Department\\
         University of Wisconsin\\
         Madison, Wisconsin 53706} 
\email{maribeff@math.wisc.edu}
\author{M. Eastwood}
\thanks{{M. E. supported by the Australian Research Council}}
\address{Centre for Mathematics and its Applications\\ Mathematical Sciences Institute\\
Australian National University\\ Canberra, ACT 0200, Australia} 
\email{Michael.Eastwood@anu.edu.au}

\date{March, 2009}
\maketitle

\begin{abstract} In this paper we relate the geometric Poisson brackets on the $2$-Grassmannian in $\R^4$ and on the $(2,2)$ M\"obius sphere. We show that, when written in terms of local moving frames, the geometric Poisson bracket on the M\"obius sphere does not restrict to the space of differential invariants of Schwarzian type. But when the concept of conformal natural frame is transported from the conformal sphere into the Grassmannian, and the Poisson bracket is written in terms of the Grassmannian natural frame, it restricts and results into either a decoupled system or a complexly coupled system of KdV equations, depending on the character of the invariants. We also show that the biHamiltonian Grassmannian geometric brackets are equivalent to the non-commutative KdV biHamiltonian structure. Both integrable systems and Hamiltonian structure can be brought back to the conformal sphere.
\end{abstract}
\section{Introduction} Given a flat homogeneous space, one can define a Hamiltonian structure
on the space of differential invariants (curvatures) of parametrized curves (\cite{M1}).  These structures are often linked to completely integrable PDEs and to their geometric realizations (invariant curve flows on the homogeneous space inducing the integrable system on its invariants). There has recently been a flurry of literature studying the existence of these integrable systems and their associated geometric flows, for example see \cite{Anco}, \cite{Doliwa}, \cite{Qu1}, \cite{Qu2}, \cite{LP}, \cite{SW}, \cite{TT}, \cite{Terng} and references within.

We will say a differential invariant $I$ is of Schwarzian type whenever $\phi^\ast I = ((\phi')^2I)\circ \phi^{-1} + S(\phi)\circ\phi^{-1}$, where $\phi^\ast I$ represents the the pull-back of $I$ by a diffeomorphism $\phi$ and where $S(\phi)$ indicates the Schwarzian derivative of $\phi$ (the Schwarzian derivative itself behaves this way under reparametrization). In \cite{M3} the author conjectured that the nature of the geometry (and its invariants) was linked to  the type of integrable systems they could realize. In particular, she conjectured that the existence of differential invariants of projective (or Schwarzian) type would result in geometric realizations of equations of KdV type (as it appeared in  \cite{M3} and \cite{M4}, for example), while the existence of invariants of Riemannian type would result in geometric realizations of nonlinear Schr\"odinger equations (NLS), modified KdV (mKdV) and Sine Gordon (as in \cite{Anco}, \cite{SW} and \cite{Terng}, for example). In particular, one can obtain a geometric realization of the KdV equation by flows in $\RP^1$, of generalized KdV equations by flows on $\RP^n$, of a system of complexly coupled KdV equations by conformal flows and of a decoupled system of KdV equations by a flow in the Lagrangian Grassmannian. Eastwood conjectured that this dichotomy might be related to the existence of preferred parametrizations, projective versus affine, as in \cite{ES}. Many of these geometric realizations are given by flows for which all non-Schwarzian invariants vanish or 
 are constant (i.e. initial curves are restricted). That is, they are in fact completely integrable level sets of associated curve flows which are not completely integrable themselves. More interestingly, the existence of these integrable level sets  is always linked to the reduction of a bi-Hamiltonian Poisson structure to the submanifold of vanishing non-Schwarzian invariants (as in \cite{M3} and \cite{M4}). 

Most of the above examples are particular cases of flat parabolic geometries, that is, homogenous spaces of the form $G/P$, $G$ semisimple and $P$ a parabolic subgroup. (In fact, they are instances of parabolic geometries associated to $|1|$-gradings of the algebra.) One such case does not seem to
behave the way the other cases do, namely the spinor case $G = O(2n,2n)$. The author showed in \cite{M3} that, even though spinor curves do possess differential invariants of Schwarzian type, the geometric Poisson structure associated to flows of spinors does not reduce to the submanifold of vanishing non-Schwarzian invariants. Furthermore, the somehow expected flow possessing a decoupled system of KdV equations as level set does not preserve this submanifold so that one cannot find a geometric realization of a system of decoupled KdV by flows of spinors.

In this paper we study the flat Grassmannian case of $2$ dimensional planes in $\R^4$. This homogeneous space can be identified with the homogeneous manifold $\SL(2+2,\R)/P$, for a properly chosen parabolic subgroup $P$. The notation $\SL(2+2)$ refers to the action of $\SL(4)$ on the manifold, as shown in Section 3. Since the manifold is flat, the Cartan connection of the manifold will be given by the Maurer-Cartan form (our results are local). As we will see, the group $\SL(2+2, \R)$ is a double cover if $\O(3,3)$ and the cover induces an equivalency of parabolic geometries. Indeed, the oriented conformal sphere $\SO(3,3)/P$ can also be viewed as the spin manifold $\mathrm{Spin}(3,3)/P$, itself isomorphic to $\SL(2+2)/P$. At the infinitesimal level the isomorphism is given by an isomorphism of the Lie algebras and their associated gradations.  Our original intention was to translate our knowledge of the conformal case into the Grassmannian. As it turned out, we also ended up learning more about the conformal case from the Grassmannian situation.

 Moving frames and differential invariants for curves in Grassmannian manifolds $\mathrm{Gr}(p,q) = \SL(p+q)/P$ where $P$ is a suitable chosen parabolic subgroup, are not well-known in general. For the case Gr$(nr,r)$, a special type of non-local invariants were found in \cite{Se}. These invariants correspond to a Laguerre-Forsyth canonical form for the Serret-Frenet equations, and we will use them at the end of the paper. In section 3 we will find a local moving frame along curves in Gr$(2,2)$ and we will find the differential invariants they generate. We will show that two of the four generating invariants are invariants of Schwarzian type. In Theorem \ref{levelset} we will find explicitly the most general form of an invariant Grassmannian flow and we will show that even those who have normalized coefficients do not preserve the submanifold of vanishing non-Schwarzian invariants; when these invariants vanish, the invariant flow blows up. This will imply that the geometric Poisson bracket does not restrict to the space where the non-Schwarzian invariants vanish, much like the situation in the spinor case \cite{M3}. This seems to be somehow counterintuitive, 
 integrable level sets do exist in the conformal sphere of signature $(n,0)$ for which we
 can find a complexly coupled system of KdV equations. 
 
We show that the Grasmannian problem lies in the choice of moving frame. For this we notice that a {\it local} (i.e. depending on the curves and its derivatives) choice of moving frame in the conformal sphere results in the same type of problem the Grassmannian case had. On the other hand, if we choose a {\it natural moving frame}, a generalization of the non-local natural Euclidean frame, then the level set is preserved and both Hamiltonian structures can be reduced. In section 5 we define natural frames for both the conformal sphere of signature $(2,2)$ and for Grassmannian curves. We then prove that we can find a Grassmannian geometric realization inducing a complexly coupled system of KdV equations on the Grassmannian curvatures of projective type. Furthermore, we show that there is also a geometric realization of a decoupled system of two KdV equations.

Finally, we show that, when written in terms of the moving frame generating the non-local invariants appearing in \cite{Se}, the biHamiltonian geometric structure {\it on the complete Grassmannian} is equal to the non-commutative KdV biHamiltonian structure. We also show that the noncommutative KdV equation has a Grassmannian geometric realization. The noncommutative KdV equation and its biHamiltonian structures were defined in \cite{OS}. Given the relation to the conformal sphere, these produce also conformal biHamiltonian structures and a geometric realization for this system. The only conformal realizations that were previously known were those of the coupled KdV system. Using the isomorphism, we also prove that the Poisson brackets on the conformal $(2,2)$ sphere are the noncommutative KdV structures.

\section{ Grassmannian-Conformal parabolic equivalence} 
\subsection{Description of the manifolds}  Let us first realize $\mathrm{Gr}_2(\R^4)$ as the homogeneous space $\SL(4,\R)/P_G$ where $P_G$ is the parabolic subgroup of $\SL(4,\R)$ defined by matrices of the form
\[
\begin{pmatrix} A & \mathbf 0 \\  C & B\end{pmatrix}
\]
where $A, B, C, \mathbf 0 \in M_{2\times2}$ and where $\det A \det B = 1$. The subindex $G$ in $P_G$ indicates its association to the Grassmannian. Its Lie algebra $\p_G$ is defined by similarly shaped matrices with vanishing trace. This quotient corresponds to a gradation of the algebra $\g = \g_1\oplus\g_0\oplus\g_{-1}$, where $\p_G = \g_1\oplus\g_0$, and $\g_{-1}$ is defined by the upper right block (the dual to $\g_1$).

We now describe the conformal sphere with signature $(2,2)$. Let $J\in M_{6\times 6}$ be the matrix
\begin{equation}\label{J}
J = \begin{pmatrix} 0&0&0&0&0&1\\ 0&0&0&0&1&0\\0&0&0&1&0&0\\0&0&1&0&0&0\\0&1&0&0&0&0\\1&0&0&0&0&0&\end{pmatrix}.
\end{equation}
We can realize the group $\SO(3,3)$ as the identity component of the group $\O(3,3)$ defined as \[\O(3,3) = \{ A \in \GL(6,\R), \text{such that}\, A^t J A = J\}.\] With this realization the Lie algebra will be given by matrices which are skew symmetric with respect to the secondary diagonal, that is, $X \in \so(3,3)$ whenever $X^tJ + JX=0$
\begin{equation}\label{so33}
X = \begin{pmatrix} x_{11}&x_{12}&x_{13}&x_{14}&x_{15}&0\\ x_{21}&x_{22}&x_{23}&x_{24}&0&-x_{15}\\ x_{31}&x_{32}&x_{33}&0&-x_{24}&-x_{14}\\ x_{41}&x_{42}&0&-x_{33}&-x_{23}&-x_{13}&\\ x_{51}&0&-x_{42}&-x_{32}&-x_{22}&-x_{12}\\ 0&-x_{51}&-x_{41}&-x_{31}&-x_{21}&-x_{11}\end{pmatrix}.
\end{equation}
Next, define $P_C$ (the subindex $C$ indicates its association with the conformal sphere) to be the parabolic subgroup of $\SO(3,3)$ given by the stabilizer of the line 
\[
\begin{pmatrix} 0\\0\\0\\0\\0\\\ast\end{pmatrix}\in \R^6
\]
under the linear action of $\SO(3,3)$; that is, the stabilizer of the basepoint in $\RP^5$ under the projective action. Its orbit is the non-singular quadric
\[
Q \equiv \{[v]\in \RP_5, \text{such that}\, v^t J v = 0\}= SO(3,3)/P_C.  
\]
The parabolic Lie algebra $\p_C$ is given by those elements in $\so(3,3)$, as in (\ref{so33}), for which $x_{1i} = 0$, $i=2,3,4,5$. As before, the quotient is related to a gradation of the algebra $\g=\g_1\oplus\g_0\oplus\g_{-1}$ with $\g_{-1}$ dual to $\g_1$ and $\p_C = \g_1\oplus\g_0$.
\subsection{Isomorphism between the homogeneous manifolds} It is well-known that there exists an isomorphism of homogeneous spaces
\[
\mathrm{Gr}_2(\R^4) \cong Q.
\]
The isomorphism is induced by a homomorphism at the Lie group level. Specifically, for $A \in \SL(4,\R)$, define $\Phi(A) \in \mathrm{Hom}(\Lambda^2 \R^4, \Lambda^2\R^4)$ by the usual induced action on simple vectors $v\wedge w$. That is
\[
\Phi(A) (v\wedge w) = Av\wedge Aw.
\]
We can identify $\Lambda^2 \R^4$ with $\R^6$ through the choice of basis
\[
\{e_1\wedge e_2, e_1\wedge e_3, e_1\wedge e_4, e_2\wedge e_3, e_4\wedge e_2, e_3\wedge e_4\}.
\]
Under this representation, $\Phi(A)$ is defined by an $\SO(3,3)$ matrix. Through straightforward calculations we see that, if $A, D\in M_{2\times2}$, $B, C\in M_{2\times2}$ with $\det B\det C = 1$, then
\begin{equation}\label{equiv1}
\Phi(\begin{pmatrix}I&A\\ \0&I\end{pmatrix}) = \begin{pmatrix}1&a_{21}&a_{22}&-a_{11}&a_{12}&\det A\\
0&1&0&0&0&-a_{12}\\0&0&1&0&0&a_{11}\\0&0&0&1&0&-a_{22}\\0&0&0&0&1&-a_{21}\\0&0&0&0&0&1\end{pmatrix}, 
\end{equation}
\begin{equation}\label{equiv12}
\Phi(\begin{pmatrix}I&\0\\D&I\end{pmatrix} = \begin{pmatrix}1&0&0&0&0&0\\ d_{12}&1&0&0&0&0\\
d_{22}&0&1&0&0&0\\ -d_{11}&0&0&1&0&0\\ d_{21}&0&0&0&1&0\\\det D&-d_{21}&d_{11}&-d_{22}&-d_{12}&1\end{pmatrix}
\end{equation}
\begin{equation}\label{equiv2}
\Phi(\begin{pmatrix}B&\0\\ \0& C\end{pmatrix}) = \begin{pmatrix}\det B&0&0&0&0&0\\
0&b_{11}c_{11}&b_{11}c_{12}&b_{12}c_{11}&-b_{12}c_{12}&0\\ 0&b_{11}c_{21}&b_{11}c_{22}&b_{12}c_{21}&-b_{12}c_{22}&0\\ 0&b_{21}c_{11}&b_{21}c_{12}&b_{22}c_{11}&-b_{22}c_{12}&0\\ 0&-b_{21}c_{21}&-b_{21}c_{22}&-b_{22}c_{21}&b_{22}c_{22}&0\\ 0&0&0&0&0&\det C\end{pmatrix}.
\end{equation}
The map $\Phi$ is a double cover of $\SO(3,3)$ by $\SL(4, \R)$, which also maps $P_G$ into $P_C$. Notice that $\Phi$ is a double cover on the parabolic subgroups, while it is one-to-one between the sections of $\SL(4,\R)/P_G$ and $\SO(3,3)/P_C$ defined by (\ref{equiv1}). Therefore, the map induces the desired isomorphism between homogeneous spaces.
Clearly $\Phi$ induces a graded map at the Lie algebra level.

\section{The local geometry of Grassmannian curves}
Let $\SL(p+p)\subset \GL(p+p)$ be the simple linear group acting on $\gl(p)$ according
to the action of $\SL(2p)$ on the homogeneous space $M = \SL(2p)/H$, where $H \subset \SL(2p)$ are matrices
of the form 
\[
\begin{pmatrix} E & 0\\ C& D\end{pmatrix}
\]
with $E, C, D \in M_{p\times p}$.
Assume we are in a neighborhood of the identity and so we can locally factor an element $g \in \SL(2p)$ into the product
\[
g = \begin{pmatrix} I & 0\\ Z & I \end{pmatrix} \begin{pmatrix} A&0\\ 0&B\end{pmatrix} \begin{pmatrix} I&Y\\ 0&I\end{pmatrix}.
\]
Then, a representation for the
homogeneous manifold $M$ is given by the section defined by matrices of the form
\[
\begin{pmatrix}I& u \\ 0&I\end{pmatrix}.
\]
Using this section we can write the $\SL(p+p)$
action on $\gl(p)$ as determined by the relation 
\[
g \begin{pmatrix} I&u\\ 0&I\end{pmatrix} = \begin{pmatrix} I &g\cdot u\\ 0&I \end{pmatrix} h
\]
where $h \in H$. This relation determines the action uniquely to be
\begin{equation}\label{action}
g\cdot u = A(u+Y)\left(B+ZA(u+Y)\right)^{-1}.
\end{equation}

\subsection{Group-based moving frames for Grassmannian generic curves}\label{sframe}
In this section we will use the normalization method described in \cite{FO} to find a group-based moving frame along
generic parametrized curves in the manifold $\mathrm{Gr}_2(\R^4)$. Let $J^{(k)}(\R, \mathrm{Gr}_2(\R^4))$ be the $k$-jet space of curves in $\mathrm{Gr}_2(\R^4)$, i.e., the set of equivalence classes of curves in $\mathrm{Gr}_2(\R^4)$ up to $k$-contact order. An $m$-order left (resp. right) group-based moving frame is a map
\[
\rho: J^{(m)}(\R, \mathrm{Gr}_2(\R^4)) \to \SL(2p)
\]
equivariant with respect to the prolonged action  of $\SL(2p)$ on $J^{(k)}(\R, \mathrm{Gr}_2(\R^4))$ and the left (resp.
right) action of $\SL(2p)$ on itself. Let $u_r = \frac{d^ru}{dx^r}$, where $x$ is the parameter. In the case at hand, the prolonged action is defined by the relation
\[
 g\cdot u_r = (g\cdot u)_r
 \]
 where the right hand side represents the formula given by differentiating $r$ times the action $g\cdot u$,  and
 writing it in terms of $x,u,u_1,\dots, u_r$, the coordinates in $J^{(k)}(\R,\mathrm{Gr}_2(\R^4))$, $r\le k$.
 
 In order to find a left moving frame $\rho$ along a generic curve $u(x)$, we will normalize the prolonged
action of $\SL(2p)$ on $J^{(m)}(\R, \mathrm{Gr}_2(\R^4))$, up to a certain order $m$. The choice of normalization (which defines a cross-section on the prolonged orbits of the group) is not unique and can be arbitrary as far as we retain full rank. If $\rho$ is a left (resp. right) moving frame, we call $K = \rho^{-1}\rho_x $ (resp. $K = \rho_x\rho^{-1}$) its associated {\it Maurer-Cartan matrix}. A theorem by Hubert (\cite{H}) states that, if $\rho$ is found via a normalization process, the entries of $K$ and their derivatives functionally generate all other differential invariants of the curve.
Different choices of normalization will give rise to different Maurer-Cartan matrices and different invariants. Our particular choices are made seeking both simplicity and a direct relation between the invariants (to this end the normalization constants are coordinated in both examples). Simplicity is important, a complicated Maurer-Cartan matrix will result in a difficult Hamiltonian study.
 
 At each step we will normalize fully,
i.e., we will normalize as many terms as permitted by the rank of the action. The terms that cannot be normalized will be differential invariants of the action. This process will determine
an element $g$ completely in terms of $u$ and its derivatives. It is known (see \cite{FO}) that $\rho^{-1} = g$ found through this process is a right moving frame. A left moving frame is given by its inverse $\rho$. 

We proceed to determine the (right) frame for the case at hand.

{\it Zeroth normalization equation}. For first normalization constant we will choose $i_0 = 0$. The first normalization
equations will be equations of zero differential order 
\[
g\cdot u =A(u+Y)\left(B+ZA(u+Y)\right)^{-1} = i_0 = 0
\]
which is satisfied by the choice $Y = -u$.  We have no zero order differential invariants.

{\it First normalization equation}. The next equations are the first order normalization equations $g\cdot u_1 = i_1$. We will make the normalization choice $i_1 = I$.
After substituting the previous normalization choice ($u+Y = 0$), the equation becomes
\begin{equation}\label{act1}
g\cdot u_1 = Au_1\left(B+ZA(u+Y)\right)^{-1} \end{equation}\[ - A(u+Y)\left(B+ZA(u+Y)\right)^{-1} Z \left(B+ZA(u+Y)\right)^{-1} = A u_1 B^{-1} = i_1 = I.
\]
This equation is satisfied with the choice 
\[
A = B u_1^{-1}.
\]
 Since $g \in \SL(2p)$, we also have
$\det A \det B = 1$ and so 
\[
\det B = (\det u_1)^{1/2}.
\] 
As expected, we have no first order differential invariants. Let us call $F = \left(B+ZA(u+Y)\right)$. 

{\it Second normalization equation}. After
differentiating again, the second order normalizations are found by substituting previous normalization values (in this case $u+Y = 0$ and $A = Bu_1^{-1}$)
and making the result equal to a constant $i_2$. In this case we choose $i_2=0$. That is
\begin{equation}\label{act2}
g\cdot u_2 = A u_2 F^{-1} - 2 A u_1F^{-1}ZAu_1 F^{-1} \end{equation}\[ - A(u+Y)F^{-1} Z\left(Au_2F^{-1} -2 A u_1F^{-1}ZAu_1 F^{-1}\right)  \]\[ = Bu_1^{-1}u_2 B^{-1} - 2 Z = i_2 = 0.
\]
This is solved with the choice 
\[
Z = \frac12 Bu_1^{-1} u_2 B^{-1}
\]
and we have no second order invariants. At this point we only  have left the determination of $B$ (although
not its determinant). 

{\it Third normalization equation}. The third order normalization equations are, after some simplification, given by
\begin{eqnarray}\label{act3}
g\cdot u_3  &=& B\left(u_1^{-1} u_3 - \frac32 u_1^{-1}u_2u_1^{-1} u_2\right) B^{-1} = i_3.
\end{eqnarray}
\begin{definition} We call 
\[
S(u) = u_1^{-1} u_3 - \frac32 u_1^{-1}u_2u_1^{-1} u_2
\]
 the {\it Schwarzian derivative of the Grassmannian curve $u$}.
 \end{definition}
Equation $BS(u)B^{-1} = i_3$  does not have full rank on $B$ for any choice of $i_3$, we can at most reduce $S(u)$ to a certain normal form under conjugation. Let us assume now that $p=2$. The action $U \to BUB^{-1}$ has two invariants, namely, the determinant and trace of $U$. That means that the rank of this action is two and we will be able to use at most two third order normalization equations. Therefore, there will
be two third order differential invariants given by the entries that cannot be further normalized. Notice that these are Grassmannian invariants of Schwarzian-type.

 The following result is a consequence of theorems that can be found, for example, in \cite{G}.
 \begin{proposition} Two generating and (functionally) independent third order differential invariants for a Grassmanian curve $u$ in $M$ are given
 by the determinant and the trace of its Schwarzian derivative. That is, any other third order differential 
 invariant of $u$ must be a function of these two.
 \end{proposition}
There are many possible choices for $i_3$. Our choice will be to normalize $i_3$ depending on the 
nature of its eigenvalues (real or complex) . Let us write
\[
S(u) = \begin{pmatrix} s_1&s_2\\ s_3&s_4\end{pmatrix}, \hskip 2ex B = \begin{pmatrix} a&b\\ c&d\end{pmatrix}
\]
in that case $i_3$ is given by
\begin{eqnarray*}
i_3 &=&  B S(u) B^{-1} \\&=& (\det u_1)^{-1/2}\begin{pmatrix}d( as_1+bs_3)-c(as_2+bs_4)&-b( as_1+bs_3)+a(as_2+bs_4)\\ d(cs_1+ds_3) - c(cs_2+ds_4)&-b(cs_1+ds_3) +a(cs_2+ds_4)\end{pmatrix}.
\end{eqnarray*}
If we call $a/b = \alpha$ and $c/d = \beta$. 

In the generic case of real eigenvalues, diagonalizing $i_3$ is equivalent to solving the same equation for both $\alpha$ and $\beta$, namely
\begin{equation}\label{alphabeta}
P(\alpha) = \alpha^2s_2 + \alpha(s_4-s_1) - s_3 = 0, \\ \hskip 2ex P(\beta) = \beta^2s_2 + \beta(s_4-s_1) - s_3 = 0.
\end{equation}
Since $\det B\ne 0$ we need this equation to have two different solutions and we need $\alpha$
and $\beta$ to be those two different solutions (the discriminant condition on this equation is
the same as the one for the characteristic equation for $S(u)$). Therefore $i_3 = \begin{pmatrix} k_1&0\\ 0&k_2\end{pmatrix}$.  

In the generic case of complex eigenvalues, the normal form will be
\begin{equation}
\begin{pmatrix} k_1&-k_2\\ k_2&k_1\end{pmatrix}\end{equation}
and the corresponding conditions to achieve it
 are
\begin{equation}\label{alphabetac}
2\alpha\beta s_2+(\beta+\alpha)(s_4-s_1)-2s_3 = 0,  \hskip 2ex {b^2}P(\alpha)= {d^2}P(\beta),
\end{equation}
where $P(\alpha)$ and $P(\beta)$ are given as in (\ref{alphabeta}).

In both cases, the condition 
\begin{equation}\label{detB}
(\alpha - \beta) bd = \det B = (\det u_1)^{1/2}
\end{equation}
 together with the two equations solve for $a, c$
and $d$ in terms of $b$,  in the real case, or for $\alpha$, $b$ and $d$ in terms of $\beta$ in the complex case. The real case is straightforward, the complex case might require a little more explanation. From (\ref{alphabetac}) we obtain
\begin{equation}\label{alphac}
\alpha = \frac{\beta(s_4-s_1)-2s_3}{s_1-s_4 - 2\beta s_2}
\end{equation}
and
\begin{equation}
b^4 = \frac14 \frac{(2\beta s_2+s_4-s_1)^2 \det u_1}{P(\alpha)P(\beta)}.
\end{equation}
Using (\ref{alphac}) one can see that
\begin{equation}\label{palpha}
P(\alpha) = -\frac{P(\beta)}{(2\beta s_2 +s_4-s_1)^2} \Delta
\end{equation}
where $\Delta = (s_4-s_1)^2 + 4 s_3s_2$ is the discriminant of the characteristic polynomial of $S(u)$, and hence negative in our case. Therefore
\[
b^4 = \frac14 \frac{-\Delta(2\beta s_2+s_4-s_1)^4 \det u_1}{P(\beta)^2}.
\]
From (\ref{alphabetac}) and (\ref{palpha}), we also get
\begin{equation}\label{d}
d^4 = -\frac14 \frac{\Delta^3 \det u_1}{P(\beta)^2}.
\end{equation}
Thus, the only condition we need to be able to solve for $b$ and $d$ is $\det u_1 >0$, which has been required early on, in view of (\ref{detB}).
 
{\it Fourth normalization equation}. Finally, if we write the fourth order normalization equations and we use previous normalization we will
obtain an equation also of the form
\begin{equation}\label{fourth}
B R(u) B^{-1} = i_4,
\end{equation}
for some fourth order matrix $R(u)$ involving derivatives of $u$, but not $B$. We do not really need here its explicit expression, but we need to choose the last normalization equation carefully so as to simplify later calculations. 

{\it Real case}. 

Assume first we are in the real case. If we denote $R(u)$ by
\[
R(u) = \begin{pmatrix} r_1& r_2\\ r_3& r_4\end{pmatrix} 
\]
then normalizing its $(1,2)$ entry, for example, will give us the remaining missing equation. Although one usually chooses constants to normalize, we can also choose expressions that depend on invariants as we simply need a section of the prolonged orbit (see \cite{FO}). The resulting generation of invariants will remain unchanged, although the generators will be different. Assume we have diagonalized the Schwarzian derivative and it is given by
\[
\begin{pmatrix} k_1&0\\ 0&k_2\end{pmatrix}.
\]
Generically $k_1\ne k_2$ and we can assume that $k_1-k_2$ will not vanish near the basepoint. Let us choose as last normalization equation to be the $(1,2)$ entry equal $k_2-k_1$
\begin{equation}\label{realnorm}
(\det u_1)^{-1/2}b^2 \left( \alpha^2 r_2 + \alpha (r_4-r_1) - r_3\right) = k_2-k_1.
\end{equation}
(Notice that zero is a singular value for the normalization.) Recall that $k_2 - k_1 = \pm (\Delta)^{1/2}$, depending on the choice. We can exchange $k_1$ and $k_2$ if locally this equation cannot be solved. (In fact, we can also choose any multiple of $k_2-k_1$ instead of $k_2-k_1$ and obtain similar results.) Thus, with a proper choice, this equation can be generically solved. 

{\it Complex case}. 

Assume now that we are in the complex case. In order to find $\beta$ we can choose several different normalizations. The one we will choose here is to make the $(1,1)$ entry of the equation equal to the $(2,2)$ entry. The corresponding equation is
\[
\alpha \beta r_2 +\frac12 (\alpha+\beta)(r_4-r_1) + r_3 = 0.
\]
After substituting the value of $\alpha$ and simplifying, we obtain a quadratic equation for $\beta$, 
which can be locally solved in generic cases. As before, $k_1 = \frac12 (s_1+s_4)$ and $k_2 =\pm\frac12 (-\Delta)^{1/2}$, depending on the choice. Different normalizations will be needed for those generic cases for which the equation cannot be solved. Later on we will need to change this frame into a more convenient one.

This completes the calculation of
the moving frame and we can now find its associated Maurer-Cartan matrix. Unlike the classical case, the entries of the  Maurer-Cartan matrix will produce a complete set of independent and generating differential invariants (see \cite{H}). But before, we will write a final description of the generating invariants that can be obtained directly from this normalization.

\begin{theorem} A generic curve in $\mathrm{Gr}_2(\R^4)$ has a system of four functionally independent and generating  differential invariants, two of order three and one of each order 4 and 5.
\end{theorem}
\begin{proof}  Standard arguments that can be found in \cite{FO} tell us that, in the case of real Schwarzian eigenvalues, the diagonal of $i_4$ in (\ref{fourth}) equals the derivative of the diagonalization of $S(u)$. In the complex case, the same
arguments show that, if
\[
i_3 = \begin{pmatrix} k_1&-k_2\\ k_2&k_1\end{pmatrix}
\]
then $i_4$ in  (\ref{fourth}) will be of the form
\[
\begin{pmatrix} \alpha_1 & \beta_1\\ \beta_1 & -\alpha_1\end{pmatrix} + \begin{pmatrix} (k_1)_x&-(k_2)_x\\ (k_2)_x&(k_1)_x\end{pmatrix}.
\]
In the real case, (\ref{fourth}) adds one independent fourth order invariant to our group of third order ones, namely the $(2,1)$ entry of (\ref{fourth}) (which has not been normalized). In the complex case, our normalization implies $\alpha_1=0$. Hence, $\beta_1$ will be the additional functionally independent fourth order invariant.  

One more fifth order invariant
exists, one that is not  a function of the previous three invariants and their derivatives. The description of this generator  can be found following the normalization process in \cite{FO}.
If we write the fifth order normalization equations as above, there is one entry that corresponds to a previously normalized entry in $i_4$, namely $(1,2)$ in the real case and $(2,1)$ in the complex one. That entry is an independent fifth order differential invariant.
\end{proof}

The explicit expression is not relevant now. Indeed we will make use of a different, but more convenient, choice.

\subsection{Grassmannian Maurer-Cartan matrix associated to the left moving frame}
\begin{definition} Let $\rho$ be a left moving frame. The matrix
\[
K = \rho^{-1} \rho_x
\]
is called the (left) Maurer-Cartan matrix associated to $\rho$ (it is the horizontal component of the pull-back by $\rho$ of the Maurer-Cartan form of the group $G$) . The entries of $K$ are clearly differential invariants; a Theorem in \cite{H} states that they  generate all other differential invariants for the curve.
Since the right moving frame associated to a left one is its inverse, the right Maurer-Cartan matrix is the negative of the left one.
\end{definition}

Some of the work in \cite{FO} provides what are normally called horizontal recurrence formulas. These formulas can be used to recurrently calculate the entries of $K$. The following Proposition is a reformulation of these recurrence formulas for the case at hand. 

\begin{proposition}\label{horeq}
Let $K$ be the left Maurer-Cartan matrix associated to a moving frame $\rho$. Assume that $\rho\cdot u_s = i_s$, $s = 0,1,2\dots$. Then
\begin{equation}\label{recurrence}
K\cdot i_s = i_{s+1} - (i_s)_x
\end{equation}
$s = 0,1,2,\dots$, where the dot in $K\cdot i_s$ represents the s$^{\mathrm th}$ infinitesimal prolonged action of the Lie algebra on the element of the jet space $i_s$.
\end{proposition}

We will use this recurrence relation to find the left Maurer-Cartan matrix associated to the left moving frame $\rho = g^{-1}$, where $g$ was found in our previous section. Denote the Maurer-Cartan matrix by
\[
K = \begin{pmatrix} K_{11} & K_{12}\\ K_{21} & K_{22} \end{pmatrix}.
\]
If
\[
V = \begin{pmatrix} V_{11}&V_{12}\\ V_{21}& V_{22}\end{pmatrix} \in \g
\]
we can calculate the infinitesimal version of (\ref{action}) to be
\[
V\cdot u = V_{11} u - u V_{22}- uV_{21}u + V_{12}.
\]
Given that $i_0 = 0$ and $i_1 = I$ (see (\ref{act1})), for $s = 0$ (\ref{recurrence}) becomes
\[
K\cdot 0 = K_{12} = I.
\]
The infinitesimal version of the action in (\ref{act1}) is given by
\[
V \cdot u^{(1)} = V_{11} u_1 - u_1 V_{22} - u_1 V_{21} u - u V_{21} u_1.
\]
Now, $i_2 = 0$, so for $s = 1$ formula (\ref{recurrence}) becomes
\[
K\cdot i_1 = K_{11} - K_{22} = i_2 - (i_1)_x = 0,\hskip 4ex K_{11} = K_{22}
\]
Again, the infinitesimal version of the action in (\ref{act2}) is given by
\[
V\cdot u^{(2)} = V_{11} u_2 - u_2 V_{22} - 2 u_1 V_{21} u_1 +R_2
\]
where $R_2$ are terms that vanish whenever $u = 0$. From here, when $s = 2$, formula (\ref{recurrence}) becomes
\[
K \cdot i_2 = -2 K_{21} = i_3, \hskip 4ex K_{21} = -\frac12 i_3.
\]
Finally, the infinitesimal version of the action in (\ref{act3}) is given by
\[
V\cdot u^{(3)} = V_{11} u_3 - u_3 V_{22} - 3(u_2V_{21}u_1+u_1V_{21}u_2) + R_3
\]
where $R_3$ vanishes whenever $u = 0$. Therefore, for $s = 3$, equation (\ref{recurrence}) becomes
\begin{equation}\label{threeK}
K_{11} i_3 - i_3 K_{11} = i_4 - (i_3)_x.
\end{equation}

{\it Real case}.

In the case of real Schwarzian eigenvalues, both sides of (\ref{threeK}) have a vanishing diagonal, and so, the off-diagonal  entries of $K_{11}$
are determined by the entries of $i_4$. That is, if 
\[
K_{21} = -\frac12 i_3 = \begin{pmatrix} \k1&0\\ 0&\k2\end{pmatrix} 
\]
then, straightforward calculations show that the off-diagonal entries of $K_{11}$ are given by
\[
\begin{pmatrix} 0&1\\ \k3&0\end{pmatrix}
\]
where $\k3$ is an independent fourth order differential invariant determined by the one appearing
in $i_4$ through (\ref{threeK}). One can find explicitly $\k3$, but we do not really need its expression here and we will change frames soon. Notice that the choice of $k_2-k_1$ in the normalization of (\ref{realnorm}) implies that the $(1,2)$ entry of $K_{11}$ is constant and equal to $1$.  Recall that zero was a singular value.

We do not need to go to the fourth order. Indeed, since $K_{11} = K_{22}$, we do know that the traces of both matrices vanish. Therefore, only one unknown entry of $K$ remains, namely the one entry determining its trace. On the other hand, we do know that the entries of $K$ generate all differential invariants, and there is one fifth order invariant that is unaccounted for. Therefore, we can conclude that the missing entry, $\k4$, is such a generator without actually calculating it explicitly. If we were to need its explicit formula, we could use the recurrence formulas further to find them. Thus, in the real case
\[
K_{11} = \begin{pmatrix} \k4&1\\ \k3&-\k4\end{pmatrix}.
\]

{\it Complex case}.

In the case of complex Schwarzian eigenvalues
\[
i_3 = \begin{pmatrix} k_1&-k_2\\ k_2&k_1\end{pmatrix}, \hskip 2ex K_{21} = -\frac12 i_3 = \begin{pmatrix} \k1&-\k2\\ \k2&\k1\end{pmatrix}.
\]
As we pointed out before,
$i_4$ can be split into
\[
i_4 = \begin{pmatrix} \alpha_1 + (k_1)_x &  -(k_2)_x\\ (k_2)_x & -\alpha_1 +(k_1)_x\end{pmatrix}
\]
since, with our normalization, $\beta_1 = 0$. If $K_{11} = (k_{ij})$, then (\ref{threeK}) becomes
\[
\begin{pmatrix} k_2(k_{12}+k_{21}) & -2k_{11}k_2\\ -2k_2k_{11} & -k_2(k_{12}+k_{21})\end{pmatrix} = \begin{pmatrix}0 & \beta_1\\ \beta_1 &0\end{pmatrix}
\]
and so
\[
K_{11} = \begin{pmatrix} \k3 & -c \\ c&-\k3 \end{pmatrix}
\]
where $c$ is still to be found and $\k3$ is a fourth order invariant obtained from $\alpha_1$. As before, $c$ needs to be the
missing fifth order generator $\k4$, since it is the only undetermined entry. Finally 
\[
K_{11}= \begin{pmatrix} \k3& -\k4 \\\k4 & -\k3\end{pmatrix}.
\]

We have finally proved the following theorem.

\begin{theorem} Let $\rho$ be the left moving frame determined in the previous section. Its Maurer-Cartan matrix is given by
\begin{equation}\label{grassK}
K = \rho^{-1} \rho_x = \begin{pmatrix} \k4&1&1&0\\ \k3&-\k4&0&1\\ \k1&0&\k4&1\\ 0&\k2&\k3&-\k4\end{pmatrix}
\end{equation}
in the generic case of real eigenvalues, and by
\begin{equation}\label{grassKc}
K = \begin{pmatrix} \k3&-\k4&1&0\\ \k4&-\k3&0&1\\ \k1&-\k2&\k3 &-\k4\\\k2&\k1&\k4&-\k3\end{pmatrix}
\end{equation}
in the generic case of complex eigenvalues. In both cases $\k1$ and $\k2$ are third order differential invariants,  $\k3$ is  fourth order and $\k4$ is  fifth order. The invariants $\k i$, $i = 1,\dots,4$ form a generating system of independent differential invariants for generic curves of Grassmannians.
\end{theorem}

\subsection{Invariant evolutions of Grassmannian curves}\label{3.3}

Once a group-based moving frame has been found, a theorem in \cite{M1} provides us with a classical moving frame (an invariant curve in the frame bundle over the curve). This classical frame can be used to write the most general form of an invariant evolution of curves in $M$, i.e., an evolution for which the group takes solutions to solutions. The following theorem explains what that evolution is in our case. The theorem is a direct consequence of the results in \cite{M1}.
\begin{theorem} Let $\rho = g^{-1}$ be the left moving frame obtained in section \ref{sframe}, and let $B$ be the matrix defining $g$. Then, the most general form of an evolution of Grassmannian curves, invariant under the action (\ref{action}), is given by
\begin{equation}\label{invev}
u_t =  u_1 B^{-1} \r B
\end{equation}
where $\r$ is any $2\times 2$ matrix depending on $\k i$, $i=1,\dots, 4$, and their derivatives.
\end{theorem}

This expression is the particular form in our example of the more general formula for $|1|$-graded Lie algebras $\rho_{-1}^{-1} \left(\rho_{-1}\right)_t = Ad(\rho_0)(\r)$, where $\r\in \g_{-1}$ is an element of $\g_{-1}$ depending on differential invariants of the curve, and where $\rho=\rho_{-1}\rho_0\rho_1$ is a left moving frame factored out following the gradation, with $\rho_{-1}$ defining our section (for more details see \cite{M1}). In our case
\[
\rho_{-1} = \begin{pmatrix} I&u\\0&I\end{pmatrix}, \hskip 2ex \rho_0 = \begin{pmatrix} u_1B^{-1}& 0\\0& B^{-1}\end{pmatrix}, \hskip 2ex \rho_1 = \begin{pmatrix} I&0\\ -Z&I\end{pmatrix}.
\]
The matrix $\r$ has a Hamiltonian interpretation that will be explained in our next section. Because
of the Hamiltonian interpretation, the following Theorems are of interest to us. They describe the behavior of evolutions as non-Schwarzian invariants $\k3, \k4$ vanish, with $\r$ having the same normal form as $S(u)$. Since $B S(u) B^{-1} = i_3 = -2 K_{12}$, the particular choice $\r = K_{12}$
corresponds to the flow
\[
u_t =  u_1 B^{-1} K_{12} B = -\frac12 u_1S(u)
\]
which is a generalization of the well-known {\it KdV-Schwarzian evolution} to Grassmannian flows.

\begin{theorem}\label{levelset} Assume $S(u)$ has real eigenvalues and
assume $\r$ is a diagonal matrix.
Then, the level set $\k3 = 0$ is invariant under evolution (\ref{invev}). 

For general diagonal $\r$, and in particular when $\r = K_{12}$, 
the level set $\k4 = 0$ is not invariant under evolution (\ref{invev}).

Assume $S(u)$ has complex eigenvalues, and assume $\r = \begin{pmatrix} a&-b\\ b&a\end{pmatrix}$ is in normal form. Then,  for general $\r$ (and in particular if $\r = K_{12}$),  the flow blows up as $\k3 \to 0$. The level set $\k4=\k3=0$
is, therefore, not preserved.
\end{theorem}
The relevance of the second part of this proposition will be better understood in our next section. The situation described in this theorem is very similar to the one for spinor curves (\cite{M3}) for which it was proved that {\it any} choice of local moving frame would induce this situation.
\begin{proof}
Let $u(t,x)$ be a flow solution of (\ref{invev}) with $\r$ diagonal. A theorem in \cite{M1} (or a straightforward calculation) shows that $N = \rho^{-1}\rho_t$ is of the form
\[
N = \begin{pmatrix} N_{11}& \r \\ N_{21}&N_{22}\end{pmatrix}.
\]
Furthermore, since $\frac d{dx}$ and $\frac d{dt}$ commute, compatibility of $\rho_x = \rho K$ and 
$\rho_t = \rho N$ imply
\[
K_ t = N_x+[K,N].
\]
This breaks up into a number of individual equations (according to the gradation of the algebra), namely \begin{eqnarray}\label{eq1}
0 &=& \r_x + N_{22} -N_{11} +\left[ K_{11}, \r\right]\\ \label{eq2}
0&=& \left(N_{11} -N_{22}\right)_x + \left[K_{11}, N_{11}-N_{22}\right]-\r K_{21} -K_{21}\r+2N_{21}\\\label{eq3}
(K_{11})_t &=& (N_{11})_x +\left[K_{11}, N_{11}\right]-\r K_{21} +N_{21}\\ \label{eq4}
(K_{21})_t
&=& (N_{21})_x+K_{21}N_{11}-N_{22}K_{21}+\left[K_{11}, N_{21}\right].
\end{eqnarray}

{\it Real case}.

Assume now that $S(u)$ has real eigenvalues and that $\r = \begin{pmatrix} r_1&0\\ 0&r_2\end{pmatrix}$ is diagonal. Let us call
\[
N_{11} = \begin{pmatrix} a_1&b_1\\ c_1&d_1\end{pmatrix}, \hskip 4ex N_{22} = \begin{pmatrix} a_2&b_2\\ c_2&d_2\end{pmatrix}.
\]
Then, using (\ref{eq1}) we obtain
\begin{equation}\label{n22n11}
N_{11}-N_{22} =\begin{pmatrix} a_1-a_2&b_1-b_2\\ c_1-c_2&d_1-d_2\end{pmatrix} = \begin{pmatrix} (r_1)_x &r_2-r_1\\ \k3(r_1-r_2)&(r_2)_x\end{pmatrix}. \end{equation}
Using (\ref{eq2}) we have
\begin{equation}\label{n21}
 N_{21} = \begin{pmatrix} -\frac12 r_1'' + \k1 r_1& (r_1-r_2)'+\k4(r_1-r_2)\\ (\frac12(\k3)'-\k4\k3)(r_1-r_2)+\k3(r_1-r_2)'& -\frac12 r_2'' + r_2 \k2\end{pmatrix}.
 \end{equation}
From here, (\ref{eq4}) becomes
\begin{equation}\label{eq5}
\begin{pmatrix} (\k 1)_t & 0\\ 0& (\k 2)_t\end{pmatrix} = \begin{pmatrix} -\frac12 r_1''' + (r_1 \k1)'+\k1r_1'& \k1b_1-\k2b_2 + \frac32(r_1-r_2)''+\k2r_2-\k1r_1\\\\ \k2 c_1 - \k1c_2& -\frac12 r_2'''+(r_2\k 2)' + \k 2r_2'\end{pmatrix} + X_1
\end{equation}
where 
\(
X_1\) is a matrix whose explicit expression can be directly found, and whose relevant property is that it is well-defined for $\k3$ and $\k4$ sufficiently small, $X_1$ becomes strictly upper triangular as $\k3 \to 0$, and  $X_1 \to 0$ as $\k3,\k4 \to 0$.

Therefore,
(\ref{eq5}) determines completely the evolution of $\k1$ and $\k2$. Notice that the choice $r_i = \k i$ produces a decoupled system of KdV equations in $\k1$ and $\k2$ in the limit. From (\ref{n22n11}) we have that, as $\k3 \to 0$,  $c_1 = c_2$, which together with (\ref{eq5}) gives us, as $\k3 \to 0$, (and as long as $\k1-\k2 \ne 0$) 
\[
c_1 = c_2 = 0.
\]
Equations (\ref{n22n11}) and (\ref{eq5}) also provide the explicit expression for $b_1$ and $b_2$, these are
\begin{eqnarray}
b_1 &=& -\frac32 \frac{(r_1-r_2)''}{\k1-\k2}+ \frac{\k1+\k2}{\k1-\k2}r_1 - 2\frac{\k2}{\k1-\k2} r_2+ Y_1
\\
b_2&=& -\frac32 \frac{(r_1-r_2)''}{\k1-\k2}+ 2\frac{\k1}{\k1-\k2} r_1- \frac{\k1+\k2}{\k1-\k2}r_2  +Y_2
\end{eqnarray}
where $Y_i \to 0$ as $\k4 \to 0$. With these values, equation (\ref{eq3}) proves the proposition. The entry $(2,1)$ of (\ref{eq3})
shows $(\k3)_t = 0$ as $\k3 \to 0$ (independently of $\k4$), and hence $\k3 = 0$ is preserved by the evolution. Finally,  the $(1,2)$ entry of (\ref{eq3}) will determined the value of $d_1-a_1$, and we know that $a_1 - a_2 = (r_1)_x$, $d_1 - d_2 = (r_2)_x$ and $a_1 + a_2 + d_1 + d_2 = 0$ (from the algebra condition). These relations completely determine all values of $N$. If we further assume that $\k4 \to 0$, after straightforward calculations the remaining values show that 
\[
(\k4)_t = \frac12 b_1'' + \frac14 (r_1'' - r_2'')
\]
and so $\k4 
 = 0$ is not, in general, preserved, even in the case $r_i = \k i$.
 
{\it Complex case}. 

Assume now that $S(u)$ has complex eigenvalues, and assume that 
 \[
 \r = \begin{pmatrix} r_1&-r_2\\ r_2&r_1\end{pmatrix}.
 \]
 Using the same equations as in the real case, we obtained the values
 \[
 N_{11}-N_{22}= \begin{pmatrix} r_1'+r_2&-r_2'-2\k3r_2\\r_2'-2\k3r_2&r_1'-r_2\end{pmatrix}
 \]
 \[
 N_{21} = \begin{pmatrix} -\frac12r_1''-r_2'+r_1\k1-r_2\k2&\frac12 r_2''-(r_1\k2+r_2\k1)\\ -\frac12r_2''+r_2\k1+r_1\k2&-\frac12r_1''+r_2'+r_1\k1-r_2\k2\end{pmatrix} + Z_1
 \]
 where
 \[
 Z_1 = \begin{pmatrix}-2\k3\k4r_2+\k3r_2&-r_2\k4+\k3(r_2'+2\k3r_2)+(r_2\k3)'\\ -r_2\k4+\k3(r_2'-2\k3r_2)+(\k3r_2)'&2\k3\k4r_2-\k3xr_2\end{pmatrix}.
 \]
 The $t$-evolution of $K_{12}$ in (\ref{eq1}) imposes some conditions on the entries of LHS, namely the $(1,1)$ and $(2,2)$ entries must be equal, and the $(1,2)$ and $(2,1)$ entries must be the negative of each other. This leads to solving for $c_1+b_1$ and for $a_1-d_1$, which are determined by
 \[
 \k2(c_1+b_1) =  z_1, \hskip 2ex \k2(a_1-d_1) = z_2
 \]
 where $z_1$ and $z_2$ are well defined as $\k3$ and $\k4$ vanish, and they both vanish in the limit. Recall that the trace of $N$ is zero. Therefore $a_1+d_1+a+2+d_2 = 2(a_1+d_1) - 2r_1' = 0$. From here
 $a_1+d_1 = r_1'$ and, in the limit, $a_1 = \frac12 r_1'$.
 From here we obtain the evolutions for $\k1$ and $\k2$, namely
 \begin{equation}\label{grasscomplex}
\begin{pmatrix} \k1\\ \k2\end{pmatrix}_t =\begin{pmatrix} -\frac12D^3 +D\k1+\k1D & -D\k2-\k2D\\
D\k2+\k2D&-\frac12 D^3 + D\k1+\k1D\end{pmatrix}\begin{pmatrix} r_1\\r_2\end{pmatrix} + \begin{pmatrix} y_1\\y_2\end{pmatrix}
 \end{equation}
where, again, $y_1$ and $y_2$ are well defined and vanish as $\k3$, $\k4$ vanish. In the limit, the evolution becomes a complexly coupled system of KdV equations, evolution that also appears in the conformal case. Still, such a limit creates a singularity in the flow. A final condition is imposed by the evolution 
of $K_{11}$ in (\ref{eq3}), where the sum of the $(1,2)$ and $(2,1)$ entries vanish. The condition gives us an equation that allows us to finally solve for $c_1$ and $b_1$ (we can solve for all the other ones
with the data we have up to this point). The equation is
\[
(c_1+b_1)'+2\k3(b_1-c_1)-2(d_1-a_1)\k4-2\k4r_2+2\k3r_2'+2(\k3r_2)' = 0.
\]
The expressions for $d_1-a_1$ and $c_1+b_1$ tells us that $b_1$ and $c_1$ are not well-defined in the limit. Therefore, the $t$-evolution of $\k4$ also blows up as $\k3 \to 0$. Finally, one can check that the level set $\k3 = 0$  is preserved and $(\k3)_t = 0$ as $\k4 \to 0$. We leave the rest of the details to the reader.
\end{proof}
\section{Geometric Hamiltonian structures}\label{Hamiltonian}
There are two well-known Poisson structures defined on the space of Loops on the dual of a semisimple Lie algebra. If $\Hop, \Fop:\Lo \g^\ast \to \R$ are two operators, their Poisson brackets are given by new operators on $\Lo\g^\ast$ defined as
\begin{equation}\label{mainbr}
\{\Hop, \Fop\}(L) = \int_{S^1} \langle B\left(\var\Hop\right)_x + \mathrm{ad}^\ast(\var{\Hop})(L), \var \Fop\rangle dx
\end{equation}
\begin{equation}\label{secbr}
\{\Hop,\Fop\}_0(L) = \int_{S^1} \langle \mathrm{ad}^\ast\left(\var \Hop\right)(L_0), \var\Fop\rangle dx
\end{equation}
where $B$ is any bilinear identification of $\g$ with its dual and where $L_0 \in \g$ is any constant element. The elements $\var\Hop, \var\Fop\in \g$ are the variational derivatives at $L$ and $\langle, \rangle$ is the invariant pairing of $\g$ with $\g^\ast$. The usual choice is to identify $\g$ with its dual $\g^\ast$ using the Killing form, or, for example the trace of the product if $\g\subset \gl(n,\R)$. With these standard choices, if $\g \subset \gl(n,\R)$, then the dual to the entry $(i,j)$ is given by the entry $(j,i)$ (or a multiple of it). 

For a given choice of normalization equations, let's denote the space of Maurer-Cartan matrices by $\K$. The author of \cite{M1} showed that, locally around a generic curve $u(x)\in G/H$ with $G$ semisimple, the space $\K$ 
 can be written as a quotient $U/\Lo H$ where $U\subset \Lo\g^\ast$ is open. She also showed that the Poisson bracket (\ref{mainbr}) could then be reduced to the quotient to produce what she called a Geometric Poisson structure, a Hamiltonian structure on the space of differential invariant of curves. The Poisson bracket (\ref{secbr}) cannot always be reduced, but its reduction usually indicates the existence of a curve evolution inducing a completely integrable system on its invariants. In this case, the integrable system is biHamiltonian with respect to both reduced brackets. The paper \cite{M1} also showed how given {\it any} evolution Hamiltonian with respect to the geometric Poisson bracket, one could always find a geometric realization as invariant flows in $G/H$. The coefficients of the realization were explicitly related to the Hamiltonian functional. 

Both the reduction of (\ref{mainbr}) and the geometric realizations of Hamiltonian evolutions can be found explicitly in most cases, under some minimal conditions. For example, the reduction process can be described in the following terms: given $h, f:\K \to \R$, we extend them to functionals $\Hop,\Fop:\Lo\g^\ast \to \R$ so that the extensions are constant on the $\Lo H$-leaves along $\K$. This condition is infinitesimally described by stating 
\begin{equation}\label{Hcond}
\left(\var\Hop\right)_x +\mathrm{ad}(K) \left(\var\Hop\right) \in \h^0,
\end{equation}
with $K \in \K$. If two such extensions can be found, then the reduced Poisson bracket is given by $\{h,f\}_R(\kb) = \{\Hop, \Fop\}(K)$, where $K$ is the Maurer-Cartan matrix defining the generating invariants $\kb$. The infinitesimal condition usually allows us to solve explicitly for $\var\Hop$.

Perhaps the most interesting part is the direct relation between invariant evolutions and evolutions that are Hamiltonian with respect to the reduction of (\ref{mainbr}). The following theorem can be found in \cite{M1}.
\begin{theorem} Assume an evolution 
\begin{equation}\label{keq}
\kb_t = F(\kb, \kb_x, \kb_{xx},\dots)
\end{equation}
is Hamiltonian with respect to the geometric Poisson bracket obtained when reducing (\ref{mainbr}) to $\K$. Assume $f$ is its Hamiltonian functional and $\F$ is an extension constant on the $\L H$-leaves along $\K$. Assume $\rho= \rho_{-1}\rho_0\rho_1$ according to a $|1|$-grading of the algebra and assume we can identify $\rho_{-1}$ with a nondegenerate curve $u$ in $G/H$. Let $\r = \left(\var\F(K)\right)_{-1}$. Then 
\begin{equation}\label{groupev}
\rho_{-1}^{-1}\left(\rho_{-1}\right)_t = Ad(\rho_0)\r
\end{equation}
is a geometric realization for (\ref{keq}).
\end{theorem}
This direct relation, together with Theorem \ref{levelset}, tells us that the reduction of (\ref{mainbr}) cannot be further restricted to the level set $\k3 = \k4 = 0$. On the other hand, 
such a reduction was possible for the conformal sphere of signature $(n,0)$, and, when reduced, one would obtain
BiHamiltonian structures for a complexly coupled system of KdV equations on $\k1$ and $\k2$ (see \cite{M3}). The level set $\k3=\k4 = 0$ was also preserved by invariant evolutions whenever $\r$ was in the same normal form as $S(u)$. To achieve this restriction in the conformal case, one needed to use {\it natural conformal moving frames}. In our next section we study the local geometry of conformal curves on the sphere of signature $(2,2)$, followed by the definition of natural frames. Natural moving frames are easier to understand geometrically in the conformal picture than in the Grassmannian one, although they correspond  algebraically, so our strategy is to shift the geometric knowledge we have in the conformal case to the Grassmannian one. 

\section{The local geometry of conformal curves}
 In this section we will describe the analogous information presented in the previous section for the case of the conformal M\"obius sphere of signature $(2,2)$. Although the case of the conformal M\"obius sphere of signature $(n,0)$ was studied in \cite{M2}, the change of signature and the need to relate it to the Grassmannian case urge us to use normalization equations, a different approach from that in the original paper. The normalization constants will be matched to the ones used for the Grassmannian case, while trying to reproduce the results in \cite{M3} for the general conformal case.  Although the
 isomorphism $\Phi$ in section 2 guarantees our process, we need to have explicit descriptions to relate it to the natural frame. 

 Following the gradation of $\so(3,3)$ we described in our second section, an element $g\in\SO(3,3)$ can be locally factored as
 \[
 g = \begin{pmatrix} 1&0&0\\ z & I & 0\\ -\frac12 \norm{z}^2& -z^t J & 1\end{pmatrix} \begin{pmatrix} \alpha & 0 & 0\\ 0&\Theta& 0 \\ 0&0&\alpha^{-1}\end{pmatrix}\begin{pmatrix} 1&-y^t J & -\frac12 \norm{y}^2\\ 0&I&y\\ 0&0&1\end{pmatrix}
 \]
 where $\alpha \in \R$, $\Theta \in O(2,2)$, $O(2,2)$ is 
represented as the group that preserves 
 \[
 J= \begin{pmatrix} 0&0&0&1\\0&0&1&0\\0&1&0&0\\1&0&0&0\end{pmatrix},
 \]
  and where $\norm{v}^2 = v^tJv$ (the notation is deceiving since it can be negative). We will also denote $\langle v, w\rangle = v^tJw$. If we identify $\SO(3,3)/P_C$ with the section $\alpha = 1, z = 0, \Theta = I$, that is, with
  \[
  g_u = \begin{pmatrix} 1&-u^t J & -\frac12 \norm{u}^2\\ 0&I&u\\ 0&0&1\end{pmatrix}
  \]
  then $g\cdot u$ is completely determined by the relation
  \[
  g g_u = g_{g\cdot u} h
  \]
  with $h \in P_C$. The relation uniquely determines
  \begin{equation}\label{confact}
  g\cdot u = \frac{\Theta(u+y)-\frac\alpha2\norm{u+y}^2z}{\alpha^{-1}-z^tJ\Theta(u+y)+\frac\alpha4\norm z^2\norm{u+y}^2}.
  \end{equation}
  \subsection{Group-based  moving frame for conformal generic curves}
  Let us denote denominator and numerator by $g\cdot u = \frac F M$. The normalization process below uses normalization constants that are coordinated with those of the Grassmannian. Consider the vectors $e_1 = (1,0,0,1)^t$, $e_2 = (0,1,1,0)^t$, $e_3 = (0,1,-1,0)^t$, $e_4 = (1,0,0,-1)^t$ so that $\norm{e_1}^2=\norm{e_2}^2 = -\norm{e_3}^2=-\norm{e_4}^2 = 2$. 
  
{\it Zeroth normalization equation}. This equation is given by 
\[
g\cdot u = i_0 = 0
\]
which can be readily solved with the choice $u+y = 0$ or $y = -u$.

{\it First normalization equation}. Since $g\cdot u_1 = \frac{F_1}M - \frac F M \frac{M_1} M$. when substituting $g\cdot u = \frac F M = 0$ we get that the first normalization equation is
\[
g\cdot u_1 = \frac{F_1} M = \alpha \Theta u_1 = i_1 = e_3.
\]
This determines $\alpha^2\norm{u_1}^2= -2$ and imposes conditions on $\Theta$. If $\norm{u_1}^2>0$ (notice the abuse of notation since $\norm{}^2$ can be negative), then this choice is not possible and both Grassmannian and conformal cases will need to be renormalized to adjust to the situation (by, for example, choosing $e_2=i_1$ and $i_1 = \begin{pmatrix} -1&0\\0&1\end{pmatrix}$  in the Grassmannian case). There is no problem doing so and we obtain an analogous result.

{\it Second normalization equation}. Since 
\[
g\cdot u_2 = \frac{F_2}M-2\frac{F_1}M\frac{M_1}M-\frac FM\left(\frac{M_2}M-2\frac{M_1^2}{M^2}\right)
\]
after substituting previous normalizations we obtain
\[
g\cdot u_2 = \alpha \Theta u_2 + 2z + 2(z^tJe_3)e_3 = i_2 = 0.
\]
This choice of $i_2$ allows us to solve for $z$ in terms of $\Theta$ (if $z = (z_1,z_2,z_3,z_4)^t$, then $(z_1,z_3,z_2,z_4)^t = -\frac12 \alpha\Theta u_2)$.
With this value for $z$, one can check directly that $\norm z^2 = \frac14\norm{u_2}^2 - \frac{\alpha^2}4\langle u_1,u_2\rangle^2$ and $z^tJe_3 = \frac \alpha 2\langle u_1,u_2\rangle$.

{\it Third normalization equation}. Given that
\[
g\cdot u_3 = \frac{F_3}M-3\frac{M_1}M\left(\frac{F_2}M-2\frac{F_1M_1}{M^2}\right)-3\frac{M_2}M\frac{F_1}M-\frac FM\left(\frac{M_3}M-6\frac{M_2M_1}{M^2}+6\frac{M_1^3}{M^3}\right)
\]
substituting previous normalizations yields the equation
\[
\alpha \Theta u_3-3\alpha^2\langle u_1,u_2\rangle z-6(z^tJe_3)^2e_3-3\norm{z}^2e_3 = i_3.
\]
If we write $z$ and $e_3$ in terms of $\Theta$ using previous normalization equations, we
arrive to the third normalization equation
\begin{equation}\label{G3}
\alpha \Theta\left(u_3-3\frac{\langle u_1,u_2\rangle}{\norm{u_1}^2}u_2+\frac32\frac{\norm{u_2}^2}{\norm{u_1}^2}u_1\right) = \alpha \Theta G_3 = i_3
\end{equation} 
where the vector $G_3$ is defined uniquely by the equation.

As expected, this equation has rank 2. In fact,
\begin{equation}\label{i1}
\langle i_3, i_1\rangle = -2\left(\frac{\langle u_1,u_3\rangle}{\norm{u_1}^2}-3\frac{\langle u_1,u_2\rangle^2}{\norm{u_1}^4}+\frac32\frac{\norm{u_2}^2}{\norm{u_1}^2}\right) = -2 I_1
\end{equation}
(we define $I_1$ through this equation) and
$
\norm{i_3}^2 = -2(I_2 + I_1^2)
$
where 
\begin{equation}\label{i2}
I_2 = \frac{\langle u_3,u_3\rangle}{\norm{u_1}^2}-6\frac{\langle u_1,u_2\rangle\langle u_2,u_3\rangle}{\norm{u_1}^4}-\frac{\langle u_1,u_3\rangle^2}{\norm{u_1}^4}+6\frac{\langle u_1,u_2\rangle^2\langle u_1,u_3\rangle}{\norm{u_1}^6}+9\frac{\langle u_1,u_2\rangle^2\langle u_2,u_2\rangle}{\norm{u_1}^6}-9\frac{\langle u_1,u_2\rangle^4}{\norm{u_1}^8}.
\end{equation}
Both $I_1$ and $I_2$ appeared in \cite{M2} for the case of signature $(n+1,1)$
and they were called in \cite{M3} {\it differential invariants of Schwarzian or projective type}. If $\phi$ is a change of parameter, one can check that 
\[
\phi^\ast I_1 = (\phi')^2 I_1\circ \phi + S(\phi), \hskip 2ex \phi^\ast I_2 = (\phi')^4 I_2\circ \phi 
\]
where $S(\phi)$ denotes the Schwarzian derivative of $\phi$. The invariant $I_2^{1/4}$ is used to define the conformal arc-length.

Now we need to make a choice for $i_3$. The value of $i_3$ in the Grassmannian case depends on whether the Schwarzian derivative of $u$ has real or complex eigenvalues. Here we will make two different choices to match the Grassmannian choices. We will also call them the real and complex cases, although this relation is only apparent through its connection to the Grassmannian.

{\it Real case}. In this case we will choose
\begin{equation}\label{reali3}
i_3 = \begin{pmatrix} 0\\\ast\\\ast\\0\end{pmatrix}.
\end{equation}
That is, we will ask $\alpha \Theta G_3$ to lie on the plane generated by $e_2$ and $e_3$. Assume $i_3 = \hat k_1e_2+\hat k_2 e_3$. Given that $\alpha \Theta u_1 = e_3$, this will determine the value of $\Theta^{-1}e_3$ and that of $\Theta^{-1}e_2$ (with $\Theta \in O(2,2)$ being guarantee by the entries that are not normalized). We need to find one more piece of information about $\Theta^{-1} e_1$ or $e_4$ to completely determine $\Theta$. We can accomplish that in the fourth normalization equation, where only one more equation needs to be normalized.

{\it Complex case}. In the complex case we will choose
\[
i_3 = \hat k_1 e_3 + \hat k_2 e_4
\]
so that we will have determined $\Theta^{-1} e_3$ and $\Theta^{-1} e_4$. Again, one more normalization in the fourth order equation will completely determine $\Theta$.

{\it Fourth normalization equation}. As done with the third equation, we will need to differentiate once more and substitute the previous normalizations. Details are dull and do not add information, so we will leave them up to the interested reader. After substituting the values for $z$ and $e_3$ in terms of $\alpha \Theta$, once more the fourth normalization equation will look like 
\[
\alpha \Theta G_4 = i_4
\]
where $G_4$ is a vector depending exclusively on derivatives of $u$. 

{\it Real case}. In this case we will ask $i_4$ to belong to the subspace generated by $e_2$, $e_3$ and $e_4$; this forces the coefficient of $e_1$ to vanish, which is the last normalization equation we need to use to complete the determination of the moving frame $\rho$. Notice that these choices can be accomplished generically (to be sure, we would need to find  $G_4$, which is straightforward but long and space consuming so we will not include it in the paper).

{\it Complex case}. In this case we will ask for the same condition on $i_4$, which will also imply that the $e_1$ coefficient of $\Theta G_4$ will vanish. This normalization determines $\Theta^{-1} e_2$ additionally to $\Theta^{-1} e_3$ and $\Theta^{-1} e_4$, so the frame is now completely determined.

 \subsection{Grassmannian Maurer-Cartan matrix associated to the moving frame} In this section we will use Proposition \ref{horeq} to find the explicit form of the Maurer-Cartan matrix K associated to the moving frame $\rho$ that we just found. Let $v \in \so(3,3)$ and assume we split it as
 \begin{equation}\label{split}
 v = \begin{pmatrix} v_\alpha& -v_{-1}^t J& 0 \\ v_1&v_0&v_{-1}\\ 0&-v_1^tJ&-v_\alpha
 \end{pmatrix}.
\end{equation}
Then, the prolonged zero, first, second and third order infinitesimal action of the algebra on the manifold are given by
\begin{eqnarray*}
v\cdot u &=& v_0\cdot u-\frac12 v_1\norm{u}^2+v_{-1},\\
v\cdot u_1 &=& v_0 u_1+v_\alpha u_1+R_1,\\
v\cdot u_2 &=& v_\alpha u_2+ v_0u_2-\norm{u_1}^2v_1+2 u_1v_1^tJu_1 + R_2,\\
v\cdot u_3 &=& v_\alpha u_3+v_0u_3-3\langle u_1,u_2\rangle v_1+3(v_1^tJu_1)u_2+3(v_1^tJu_2)u_1 +R_3,
\end{eqnarray*}
where $R_i$ all contain terms that vanish as $u$ vanishes. Let $K$ be the Maurer-Cartan matrix in the conformal case and assume we split $K$ as done in (\ref{split}). Using these equations for the infinitesimal action in Proposition \ref{horeq} we get the following  information on $K$
\begin{eqnarray*}
K\cdot i_0 &=& K_{-1} = i_1= e_3,\\
K\cdot i_1 &=& K_0 i_1+K_\alpha i_1 = (K_0+K_{\alpha})e_3 = 0,\\
K\cdot i_2 &=& 2K_1-2(K_1^tJe_3) e_3 = i_3,\\
K\cdot i_3 &=& K_0 i_3+K_\alpha i_3 = i_4-(i_3)_x.\\
\end{eqnarray*}
The first equation determines $K_{-1}$ and the second tells us that $K_\alpha = 0$ and 
\begin{equation}\label{K_0}
K_0 = \begin{pmatrix} a&b&b&0\\c&0&0&-b\\c&0&0&-b\\0&-c&-c&-a\end{pmatrix}
\end{equation}
where the entries still need to be determined. The last equation will determine the value of $K_1$, depending on whether we are in a real or a complex case, as expected. In the real case, 
if $i_3 = \hat k_1 e_2+\hat k_2 e_3$, then 
\[
K_1 = \begin{pmatrix} 0\\k_1\\k_2\\0\end{pmatrix}
\]
where $k_1 = \frac12(\hat k_1-\hat k_2)$ and $k_2 = \frac12(\hat k_1+\hat k_2)$.

In the complex case, if $i_3 = \hat k_1 e_3+\hat k_2 e_4$, then
\[
K_1 = \begin{pmatrix} -k_2\\k_1\\-k_1\\k_2\end{pmatrix}
\]
where $k_i = -\frac12 \hat k_i$. These are combinations of the classical invariants $I_1, I_2$.

Our last recurrence relation completely determines $K_0$. Indeed, in the real case $K_0 i_3 = i_4-(i_3)_x$ implies that $K_0 i_3$ has vanishing $e_1$ component. If $K_0$ is given as in (\ref{K_0}), this condition implies $c = b= k_4$ and $a = k_3$. In the complex case, $K_0 i_3 = i_4-(i_3)_x$ implies also that the $e_1$ component of $K_0 e_3$ vanishes, which in this case implies $a = 0$. Notice that one can go back in our calculations and obtain an explicit formula for these invariants, if needed. We just proved the following theorem.
\begin{theorem}\label{MauCmat} Let $\rho$ be the left conformal moving frame found in the previous subsection. Then $\rho^{-1} \rho_x = K$ where $K$ is of the form
\[
K= \begin{pmatrix} 0&0&1&-1&0&0\\0&k_3&k_4&k_4&0&0\\k_1&k_4&0&0&-k_4&1\\ k_2&k_4&0&0&-k_4&-1\\0&0&-k_4&-k_4&-k_3&0\\ 0&0&-k_2&-k_1&0&0\end{pmatrix}
\]
in the real case, and
\[
 K= \begin{pmatrix} 0&0&1&-1&0&0\\-k_2&0&k_3&k_3&0&0\\k_1&k_4&0&0&-k_3&1\\ -k_1&k_4&0&0&-k_3&-1\\k_2&0&-k_4&-k_4&0&0\\ 0&-k_2&k_1&-k_1&k_2&0\end{pmatrix}
 \]
 in the complex case. The invariants $k_1$ and $k_2$ generate all differential invariants of third order, while $k_3$ and $k_4$ generate independent invariants of higher order. Explicit algebraic formulas can be found for each one of the invariants in terms of derivatives of $u$. 
\end{theorem}

 \section{Natural frames for conformal and Grassmannian curves}
 
 The concept of a natural frame appeared in \cite{B}. The idea of the author was that, although widely used, classical Serret-Frenet frames in Riemannian geometry were, by no means, the only way to frame Riemannian curves, and often not the most convenient one. The author of the article showed that one could, for example, find a classical moving frame (which he called {\it natural moving frame}) for which the derivative of any of the vectors in the frame, other than the unit tangent, will be in the tangential direction (thus, it is often called the {\it parallel frame}). Physically, the moving frame does not record any rotational movement on the normal plane. That is, if $\nu = (T, T_1, \dots T_{n-1})$ is a classical Serret-Frenet moving frame, $T$ the unit tangent, then
 \[
 \nu_x = \nu \begin{pmatrix} 0&-\kappa_1&0&\dots& 0\\\kappa_1&0&-\kappa_2&\dots&0\\ \vdots &\ddots&\ddots&\ddots&\vdots\\0&\dots&\kappa_{n-2}&0&-\kappa_{n-1}
 \\ 0&\dots&0&\kappa_{n-1}&0\end{pmatrix}
 \]
 while if $\eta = (T, N_1, \dots N_{n-1})$ is the natural frame, then
 \[
 \eta_x = \eta \begin{pmatrix} 0&-k_1&-k_2&\dots& -k_{n_1}\\k_1&0&0&\dots&0\\ \vdots &\vdots&\vdots&\vdots&\vdots\\ k_{n-1}&0&0&\dots&0\end{pmatrix}
 \]
 A main characteristic of this natural frame is that their vectors are {\it non-local}, that is, they cannot be written as an algebraic expression of derivatives of the curve. In fact, one needs to solve a linear differential equation to find the natural frame.
 
 The concept of Riemannian natural moving frame was translated into the conformal picture in \cite{M2}. In the case of the M\"obius sphere of signature $(n+1,1)$ one could find a classical moving frame $(T,T_1,\dots,T_{n-1})$ using a normalizing Riemannian-type process (see \cite{M3}). In this process, and unlike the Riemannian case, $T$ had no role and the role of generating the rest of the frame though differentiation was carried out by $T_1$ instead of the tangent ($T_1$ is a vector of differential order 3 related to $G_3$ in (\ref{G3})). In the conformal case the author of \cite{M3} defined the classical {\it conformal natural moving frame} to be the frame for which the derivatives of all vectors other than $T$ and $T_1$ are conformally in the direction of $T_1$. In terms of the group-based Maurer-Cartan matrix, if $\rho_C$ is the moving frame associated to the classical Serret-Frenet frame $(T,T_1,\dots,T_{n-1})$ as in \cite{M3} and $\rho_N$ is the one associated to $(T,N_1,\dots,N_{n-1})$, the natural frame, then
 \[
 (\rho_C)_x = \rho_C\begin{pmatrix} 0&1&0&0&0&\dots&0\\k_1&0&0&0&\dots&0&1\\ k_2&0&0&-k_3&0&\dots&0\\
 0&0&k_3&0&-k_4&\dots&0\\ \vdots&\vdots&\ddots&\ddots&\ddots&\vdots&\vdots\\0&0&\dots&\dots&0&-k_n&0\\
 0&0&\dots&0&k_{n}&0&0\\ 0&k_1&k_2&0&\dots&0&0\end{pmatrix}
 \]
 while
 \begin{equation}\label{confn0}
 (\rho_N)_x = \rho_N\begin{pmatrix} 0&1&0&0&0&\dots&0\\k_1&0&0&0&\dots&0&1\\ k_2&0&0&-\kappa_3&\dots&-\kappa_n&0\\
 0&0&\kappa_3&0&0&\dots&0\\ \vdots&\vdots&\vdots&\vdots&\vdots&\vdots&\vdots\\
 0&0&\kappa_n&0&\dots&0&0\\ 0&k_1&k_2&0&\dots&0&0\end{pmatrix}.
 \end{equation}
 As in the Riemannian case, the conformal natural frame is non-local and one needs to solve a linear differential equation to find it. More interestingly, the author of \cite{M3} showed that if the geometric Poisson bracket is to be written using a classical Serret-Frenet equation, the Poisson structure does not restrict to the submanifold $k_i=0$, $i=3,\dots$, while when written in terms of natural ones the restriction to $\k i=0$, $i=3,\dots$, can be carried out to obtain a complexly coupled system of KdV structures. The restriction is biHamiltonian and one indeed obtains a geometric realization of a complexly coupled system of KdV equations for $k_1$ and $k_2$ associated to this restriction. Our hope is that, once we find the equivalent definition of natural frame for our conformal and Grassmannian manifolds, we will also be able to obtain restrictions and realizations for both cases.
 
The definition of natural frames is not well-suited to our algebraic approach of describing moving frames. Therefore, we will try to describe natural frames algebraically in a way that can be  readily applied to our situation. Notice that, if one has a left group-based moving frame, the factor of the frame that acts linearly on the manifold (in our case $\rho_0$) represents a classical frame (see \cite{M1}). The map $M \to M$ taking $u$ to $\rho_0\cdot u$ is linear and so represented by an element of $\GL(n,\R)$. That element has in its columns a classical moving frame. As shown in \cite{M2}, the way these elements behave under differentiation will be reflected in $K_0$, according to the gradation.

Let's first have a good look at the conformal $(n+1,1)$ case, the normalization constants for the frame $\rho_N$ in (\ref{confn0}) are $i_1 = (1,\dots,0)^t$ and $i_3 =(k_1,k_2,0,\dots,0)^t$ (they correspond to the $K_{-1}$ and $K_1$ components of $K$, respectively). If $\Theta$ is the analogous to ours for the $(n+1,1)$ signature, this choice of constant implies that $\Theta u_1 = \norm{u_1}^2 e_1$ and $\Theta G_3 = k_1 e_1+k_2 e_2$; that is, the first column of the classical moving frame $\alpha^{-1} \Theta^{-1}$ will be the unit tangent and the second a combination of the tangent with $G_3$. The fact that all other derivatives  are in the direction of the second vector is reflected in the form of $K_0$, namely
\[
K_0 = \begin{pmatrix}0&0&0&\dots&0\\ 0&0&-\kappa_3&\dots&-\kappa_n\\ 0&\kappa_3&0&\dots&0\\
\vdots&\vdots&\vdots&\dots&\vdots\\ 0&\kappa_n&0&\dots&0\end{pmatrix}.
\]
Algebraically, $K_0 e_i$ is a multiple of $e_2$ for any $i=3,\dots$. 

In our signature $(2,2)$ conformal sphere $i_1 = e_3$ and, for the real case $i_3= \hat k_1 e_3+\hat k_2 e_2$ and for the complex case $i_3= \hat k_1 e_3+\hat k_2 e_4$. Therefore, in the real case we will say a moving frame is natural if its associated matrix $K_0$ satisfies $K_0 e_3 = 0$ and $K_0 e_i$ is a multiple of $e_2$ for $i=1,4$. In the complex case, we substitute $e_2$ with $e_4$. Condition $K_0 e_3 = 0$ implies 
\[
K_0 = \begin{pmatrix} a&b&b&0\\c&0&0&-b\\c&0&0&-b\\0&-c&-c&-a\end{pmatrix}
\]
while 
\[
K_0 e_1 = \begin{pmatrix}a\\ c-b\\ c-b\\-a\end{pmatrix}, \hskip 2ex K_0e_2 = \begin{pmatrix} 2b\\0\\0\\-2c\end{pmatrix}, \hskip 2ex K_0 e_4 = \begin{pmatrix}a\\ c+b\\ c+b\\a\end{pmatrix}.
\]
From here, in the real case we need $a = 0$, while in the complex case we need $b = c$.
\begin{theorem}
In the real case, there exists a conformal moving frame $\rho_N$, one we will call natural moving frame, such that
\[
K_N = \rho_N^{-1}(\rho_N)_x = \begin{pmatrix} 0&0&1&-1&0&0\\ 0&0&\kappa_3&\kappa_3&0&0\\ k_1&\kappa_4&0&0&-\kappa_3&1\\ \kappa_2&\kappa_4&0&0&-\kappa_3&-1\\ 0&0&-\kappa_4&-\kappa_4&0&0\\ 0&0&-k_2&-k_1&0&0\end{pmatrix}.
\]
In the complex case, the natural moving frame also exists and it is given by a matrix of the form
\[
K_N= \rho_N^{-1}(\rho_N)_x = \begin{pmatrix} 0&0&1&-1&0&0\\ 0&\kappa_3&\kappa_4&\kappa_4&0&0\\ k_1&\kappa_4&0&0&-\kappa_4&1\\ \kappa_2&\kappa_4&0&0&-\kappa_4&-1\\ 0&0&-\kappa_4&-\kappa_4&-\kappa_3&0\\ 0&0&-k_2&-k_1&0&0\end{pmatrix}.
\]

\end{theorem}
\begin{proof}
Given any two left moving frames $\rho_1$ and $\rho_2$, $\rho_1 = \rho_2g$ with $g$ depending on differential invariants ($g = \rho_2^{-1}\rho_1$). If $K_1$ and $K_2$ are their two Maurer-Cartan matrices, then they are related by the gauge
\[
g^{-1}g_x+g^{-1}K_2g = K_1.
\]
Therefore, we need to find an invariant $g$ gauging $K$ as in Theorem \ref{MauCmat} to $K_N$ above. One can directly check that, in the real case, $\rho_N = \rho g$, where
\[
g = \begin{pmatrix} \gamma &0&0&0\\0&1&0&0\\0&0&1&0\\0&0&0&\gamma^{-1}\end{pmatrix}
\]
and where $\gamma$ is the solution to the differential equation $\gamma _x = -k_3 \gamma$, that is, $\gamma = e^{-\int k_3}$. The invariants $k_3$ and $k_4$ change accordingly into $\kappa_3$ and $\kappa_4$.

In the complex case it is only slightly more complicated:
\[
g = \begin{pmatrix} \gamma&\eta&\eta&\gamma-1\\ -\eta&\gamma&\gamma-1&-\eta\\-\eta&\gamma-1&\gamma&-\eta\\\gamma-1&\eta&\eta&\gamma\end{pmatrix}
\]
with $1+\frac\gamma\eta\frac{\gamma-1}\eta=0$ and $\frac\gamma\eta = -\tan\left(\frac12\int(k_3+k_4)\right)$. Again, the invariants $k_4$ and $k_3$ transform accordingly. 
\end{proof}
These two transformations $k_3, k_4 \to \kappa_3, \kappa_4$ are generalizations of the well-known {\it Hasimoto transformation} for Euclidean geometry to the conformal case.  The Hasimoto transformation was proven to be a map from classical Frenet to natural moving frames (see \cite{LP1}). 

These moving frames can be translated into the Grassmannian picture using the isomorphism. For completeness we will describe natural frames and how to find them in the 
Grassmannian picture also. Notice that, even though in the conformal case one can think of geometric ways to define natural frames, this is far less intuitive and not at all obvious in the Grassmannian manifold. It is an interesting question whether or not this can be done for any plat parabolic manifold.

\begin{theorem} 
In the real case, there exists a left Grassmannian moving frame (the natural frame) such its Maurer-Cartan matrix is given by
\[
K_N = \rho_N^{-1}(\rho_N)_x=\begin{pmatrix} 0&\kappa_3&1&0\\\kappa_4&0&0&1\\ k_1&0&0&\kappa_3\\0&k_2&\kappa_4&0\end{pmatrix}.
\]
In the complex case, the natural moving frame also exists and it satisfies
\[
K_N = \rho_N^{-1}(\rho_N)_x=\begin{pmatrix} \kappa_3&\kappa_4&1&0\\\kappa_4&-\kappa_3&0&1\\ k_1&-k_2&\kappa_3&\kappa_4\\k_2&k_1&\kappa_4&-\kappa_3\end{pmatrix}.
\]

\end{theorem}
We do not need to prove this theorem, it is true due to the existence of the isomorphism. Still, we can explicitly find the relation to the moving frame we previously found. In the real case, $\rho_N = \rho g$ where
\[
g = \begin{pmatrix} \gamma&0&0&0\\0&\gamma^{-1}&0&0\\ 0&0&\gamma&0\\ 0&0&0&\gamma^{-1}\end{pmatrix}
\]
with $\gamma = e^{-\int k_4}$. In the complex case
\[
g = \begin{pmatrix} \cos\theta&-\sin\theta&0&0\\ \sin\theta&\cos\theta&0&0\\ 0&0&\cos\theta&-\sin\theta\\0&0&\sin\theta&\cos\theta\end{pmatrix}
\]
where $\theta = -\int k_4$.

\section{Invariant evolutions and biHamiltonian equations in terms of natural frames: coupled and decoupled systems of KdV equations}
 Finally, in this section we will show that both Hamiltonian structures (\ref{mainbr}) and (\ref{secbr}) reduce to the submanifold of $\K$ defined by $\k3 = \k4 = 0$ to produce a biHamiltonian decoupled system of two KdV structures in the real case, and a biHamiltonian complexly coupled system of KdV equations in the complex case. This result also implies the existence of geometric realizations for these completely integrable systems. These geometric realizations are best understood as a limit case: as $\k3, \k4 \to 0$, the evolution of $k_1$ and $k_2$ become completely integrable. Indeed, $\k3, \k4$, in principle, cannot geometrically vanish, as seen in the previous definition of natural frame (in the real case, $k_3$ will need to blow up for $\k3$ to vanish). This is a surprising geometric realization of KdV systems in the Grassmannian case, but it also shows an unexpected result, namely the existence of geometric realizations of decoupled systems of KdV equations in the conformal case. Only the one for complexly coupled systems of KdV was 
 previously known. This realization can be readily found in a general conformal sphere. For this, we merely need to choose the appropriate normalization $i_3$ as in (\ref{reali3}) in the work found in \cite{M3}. That leads to the appropriate $K_1$-form for the matrix.
 
From now on we will assume that we are working with the natural Maurer-Cartan matrices. To describe the reduction and the reduced brackets we will follow section \ref{Hamiltonian}. We will also assume that we are identifying the dual to the Lie algebra with the Lie algebra itself using the trace. That means that the dual to the entry $a_{ij}$ is the entry $a_{ji}$. In this section both the reduction of the structures, and their relation to geometric realization are explained. Assume we have a Hamiltonian functional $f:\K \to \R$ defined on the space of Grassmannian differential invariants, i.e., in the space $C^\infty(S^1)\times C^\infty(S^1)\times C^\infty(S^1)\times C^\infty(S^1)$. Assume $\F$ is an extension which is constant on the leaves of the subgroup $N$, and
assume 
 \begin{equation}\label{F}
 \frac{\partial \F}{\partial L}(K) = \begin{pmatrix} F_2&F_{-1}\\ F_1&F_3\end{pmatrix}, \hskip 2ex K = \begin{pmatrix} K_2&I\\ K_{1} & K_2 \end{pmatrix}
 \end{equation}

{\it Real case}.
In this case $K_1 = \begin{pmatrix} k_1&0\\0&k_2\end{pmatrix}$ and $K_2 = \begin{pmatrix} 0&\k3\\\k4&0\end{pmatrix}$. If we denote $f_i = \frac{\partial f}{\partial k_i}(\kb)$, then,  
\begin{equation}\label{Fdata}
F_{-1} = \begin{pmatrix} f_1&\alpha\\\beta&f_2\end{pmatrix}, F_2 = \begin{pmatrix}a&\frac12f_4+f\\ \frac12f_3+e&b\end{pmatrix}, F_3 = \begin{pmatrix} c&\frac12 f_4-f\\ \frac12f_3-e&-a-b-c\end{pmatrix}
\end{equation}
and $F_1$ is arbritary. In this case, the parabolic subgroup $P_G$ is formed by matrices of the block form $\begin{pmatrix} \ast&0\\ \ast&\ast\end{pmatrix}$. Therefore, $\n^0=\p_G^0$ can be identified with matrices with block-form
\[
\begin{pmatrix} 0&0\\ \ast&0\end{pmatrix}.
\]
Substituting (\ref{F}) in (\ref{Hcond}) we get
\begin{eqnarray}
\label{1} F_{-1}'+F_3-F_2 + [K_2,F_{-1}] &=& 0,\\
\label{2} F_2'+[K_2,F_2]-F_{-1}K_1 + F_1 &=&0\\
\label{3} F_3'+[K_3,F_3]+K_1F_{-1}-F_1 &=&0
\end{eqnarray} 
We can readily solve for $F_2-F_3$ using (\ref{1}) and for $F_1$ using (\ref{2}) minus (\ref{3}). The rest of the entries are uniquely determined by the entries in which $f'_i$ are located in the equations. Solving we get
\begin{eqnarray*}
\alpha &=& \frac1{k_2-k_1}\left(f_4'-2\k3D^{-1}(\k4f_4-\k3f_3)\right)\\
\beta &=&\frac1{k_2-k_1}\left(f_3'+2\k4D^{-1}(\k4f_4-\k3f_3)\right)\\
2a &=& f_1'+\k3 \beta - \alpha \k4+D^{-1}(\k4f_4-\k3f_3)\\
2c &=& -f_1'-\k3 \beta + \alpha \k4+D^{-1}(\k4f_4-\k3f_3)\\
2b &=& f_2'- \k3\beta +\alpha \k4 -D^{-1}(\k4f_4-\k3f_3)\\
2f &=& \alpha'+\k3(f_2-f_1)\\
2e &=& \beta'+\k4(h_1-h_2)
\end{eqnarray*}
Using this results we can calculate the reduction of (\ref{mainbr}) to the space of differential invariants. If $f$ and $h$ are two such functionals, then the Geometric Hamiltonian structure is defined as
\[
\{f,h\}(\kb) = \int_{S^1} \mathrm{trace}\left(\left(\frac{\delta \F}{\delta L}(K)\right)'+ \left[K,\frac{\delta \F}{\delta L}(K)\right]\right) \frac{\delta\Hop}{\delta L}(K)dx
\]
where $\frac{\delta\Hop}{\delta L}(K)$ is calculated similarly to the variational derivative of $\F$. The author of \cite{M1} proved that this bracket is always a Poisson bracket. To
show that the bracket can be further restricted to the submanifold $\k3=\k4=0$ we will check that, if $f$ depends only on $k_1$ and $k_2$ (that is, $f_3=f_4=0$), while $h$ depends only on $\k3$ and $\k4$ (that is, $h_2=h_1=0$), then $\{f,h\}(\kb) = 0$ along the submanifold $\k3=\k4=0$.  After this step, we can calculate the bracket of two functionals that depend only on $k_1$ and $k_2$. 

First of all, notice that condition (\ref{Hcond}) implies that
\[
\{f,h\}(\kb) = \int_{S^1}\mathrm{trace}\left((F_1'+K_1F_2-F_3K_1)H_{-1}\right)dx.
\]
If a functional $f$ satisfies $f_3=f_4=0$, then  we have $\alpha = \beta=e=f=0$ and $ a = \frac12 f_1' + C$, $b = \frac12 f_2' - C$, $c = -\frac12 f_1'+C$, where $C$ is a possible constant that will end up canceling out in the calculation of the bracket. These values imply that, whenever $\k3=\k4=0$, the matrix $F_1'+K_1F_2-F_3K_1$ will be diagonal. On the other hand, if $h_1=h_2=0$, then $H_{-1}$ will have a vanishing diagonal and, therefore, $\{f,h\}(\kb) = 0$.

Finally, if both $f$ and $h$ depend on $k_1, k_2$ only, then, when $\k3=\k4=0$
\[
F_1'+K_1F_2-F_3K_1 = \begin{pmatrix} -\frac12 f_1''+k_1f_1&0\\0&-\frac12 f_2'' + k_2f_2\end{pmatrix}_x + \begin{pmatrix}k_1 f_1'&0\\0&k_2 f_2'\end{pmatrix},
\]
and therefore
\[
\{f,h\}(\kb) = \int_{S^1}\mathrm{trace}\left(F_1'+K_1F_2-F_3K_1\right)H_{-1} dx =\int_{S^1}\begin{pmatrix}\frac{\delta f}{\delta k_1}& \frac{\delta f}{\delta k_2}\end{pmatrix}\P \begin{pmatrix}\frac{\delta h}{\delta k_1}\\ \frac{\delta h}{\delta k_2}\end{pmatrix} dx
\]
where
\[
\P = \begin{pmatrix} -\frac12 D^3 +k_1 D+Dk_1&0\\ 0& -\frac12 D^3+k_2D+Dk_2\end{pmatrix}
\]
which defines a decoupled Hamiltonian structure for KdV. One can also check directly that, if we substitute the values of (\ref{F}) in (\ref{secbr}) for the choice 
\[
L_0 = \begin{pmatrix} 0&0\\ I&0\end{pmatrix}
\]
the resulting bracket is given by
\[
\{f,h\}_0(k_1,k_2) = \frac12 \int_{S^1}\begin{pmatrix}\frac{\delta f}{\delta k_1}& \frac{\delta f}{\delta k_2}\end{pmatrix}\P_0 \begin{pmatrix}\frac{\delta h}{\delta k_1}\\ \frac{\delta h}{\delta k_2}\end{pmatrix} dx 
\]
where $\P_0$ is the second Hamiltonian structure for a decoupled system of KdV equations, that is
\[
\P_0 = \begin{pmatrix} D&0\\ 0&D\end{pmatrix}.
\] 
One can also check that the bracket of two Hamiltonians depending on $(k_1,k_2)$ only and $(\k3,\k4)$ only vanishes when $\k3$ and $\k4$ do. 

{\it Complex case}. In the complex case the same procedure is followed but with different matrices $K_1$, $K_2$, $F_1$, $F_{-1}$, $F_2$ and $F_3$. In this case
\[
K_1 = \begin{pmatrix} k_1&-k_2\\ k_2&k_1\end{pmatrix}, \hskip 2ex K_2 = \begin{pmatrix} \k3&\k4\\ -\k3&\k4\end{pmatrix}
\]
and
\[
F_{-1} = \begin{pmatrix}\frac12 f_1+\alpha&\frac12 f_2+\beta\\ -\frac12 f_2+\beta&\frac12 f_1-\alpha\end{pmatrix}
\]
\[
F_2 = \begin{pmatrix} \frac14f_3 + a&\frac14 f_4+f\\ \frac14 f_4+b&-\frac14 f_3+c\end{pmatrix}, \hskip 2ex F_3 = \begin{pmatrix} \frac14f_3 - a&\frac14 f_4+e\\ \frac14 f_4+d&-\frac14 f_3-c\end{pmatrix}.
\]
In this case, condition (\ref{Hcond}) has the same equations as the real case, but this result in different values for our unknowns
\begin{eqnarray*}
\alpha &=& -\frac1{4k_2}f_4'+\frac{\k3}{k_2}D^{-1}(\k3f_4-\k4f_3)\\
\beta &=& \frac1{4k_2}f_3'+\frac{\k4}{k_2}D^{-1}(\k3f_4-\k4f_3)\\
2a &=&  \alpha'+\frac 12 f_1'-\k4f_2\\
2c &=& - \alpha'+\frac 12 f_1'+\k4f_2\\
b+d &=& D^{-1}(\k3f_4-\k4f_3)\\
b-d &=& \beta'-\frac12f_2'+\k3(f_2-2\beta)+2\k4\alpha\\
f+e &=& -D^{-1}(\k3f_4-\k4f_3)\\
f-e &=& \beta'+\frac12f_2'+\k3(f_2+2\beta)-2\k4\alpha.
\end{eqnarray*}
 As before, if $f_3=f_4=0$ and $h_1=h_2=0$, one can check straightforward that $\{f,h\}(k_1,k_2) = 0$. The calculations are only slightly longer than the ones for the real case. Also, if both $f$ and $h$ satisfy $f_3=f_4=h_3=h_4 = 0$, then
 \[
 H_1'+K_1H_2-H_3K_1 
 \]
 \[
 = \frac12\begin{pmatrix}-\frac12 f_1'''+(k_1f_1)'+k_1f_1'+(k_2f_2)'+k_2f_2'&-\frac12 f_2'''+(k_1f_2)'+k_1f_2'-(k_2f_1)'-k_2f_1'\\
 \frac12 f_2'''-(k_1f_2)'-k_1f_2'+(k_2f_1)'+k_2f_1'&-\frac12 f_1'''+(k_1f_1)'+k_1f_1'+(k_2f_2)'+k_2f_2'\end{pmatrix}.
 \]
 Putting all of these values together in (\ref{mainbr}) we obtain the restricted bracket to be
 \[
 \{f,h\}(k_1, k_2) = \int_{S^1}\begin{pmatrix}\frac{\delta f}{\delta k_1}& \frac{\delta f}{\delta k_2}\end{pmatrix}\P \begin{pmatrix}\frac{\delta h}{\delta k_1}\\ \frac{\delta h}{\delta k_2}\end{pmatrix} dx
\]
where
\[
\P = \frac12\begin{pmatrix} -\frac12D^3+Dk_1+k_1D & Dk_2+k_2D\\ Dk_2+k_2D& \frac12 D^3-Dk_1-k_1D\end{pmatrix}
\]
the well-known Hamiltonian structure for a complexly coupled system of KdV 
equations. This structure had been already obtained in \cite{M3}. Accordingly, we can guess that the choice of third normalization in \cite{M3} is in a different prolonged orbit that the real case we found here, and that a choice in the same orbit will result in a decoupled system of KdV equations. Once more, if we substitute our values in (\ref{secbr}) for the same choice of $L_0$ as in the real case, we obtain a second Poisson structure $\P_0$, namely
\[
\P_0 = \begin{pmatrix} D&0\\ 0&-D\end{pmatrix}.
\]
 This is known to be the second Hamiltonian structure for a complexly coupled KdV system.
As before, one can also check that the bracket of Hamiltonians depending on $(k_1,k_2)$ only and $(\k3,\k4)$ only vanishes when $\k3$ and $\k4$ do. 

The interest of these different structures is that (\ref{mainbr}) reduces always to a Geometric Poisson bracket that is directly linked to a geometric realization of the Hamiltonian evolution. The relation is given as in (\ref{groupev}). Therefore, we have almost finished the proof of the following Theorem.
\begin{theorem}
Assume $u(t,x)$ describes an evolution of a curve of Grassmannian planes in $\R^4$ solution of the equation
\begin{equation}\label{Kdv-Sch}
u_t =  u_1 S(u) = u_3 - \frac32 u_2 u_1^{-1} u_2.
\end{equation} 

Then, if $S(u)$ has real eigenvalues $k_1$ and $k_2$, as $\k3, \k4 \to 0$, the curvatures $k_1$ and $k_2$ satisfy a decoupled system of KdV equations. If $S(u)$ has complex eigenvalues, $k_1\pm k_2 i$, then $k_1$ and $k_2$ satisfy a complexly coupled system of KdV equations. In both cases, the Poisson structures (\ref{mainbr}) and (\ref{secbr}) reduce to the space $\k3 = \k4 = 0$ to produce a biHamiltonian pencil for decoupled KdVs or complexly coupled KdVs, depending on the case.
\end{theorem}

\begin{proof}
Assume $\left(\frac{\delta \Hop}{\delta L}(K)\right)_{-1} = g_r$, with $g_r =\begin{pmatrix} 0&\rr\\0&0\end{pmatrix}$, $\rr$ a matrix of differential invariants, and 
where $\Hop$ is the appropriate extension of a Hamiltonian functional $h$, and where by $()_{-1}$ we indicate the projection on the tangent to the manifold $G/P$, as represented by the subspace $\g_{-1}$ of the Lie algebra $\g$. Then, from \cite{M1}, the evolution 
\[
\rho_{-1}^{-1}(\rho_{-1})_t = Ad(\rho_0) g_r
\]
induces the $h$-hamiltonian evolution on the differential invariants. If we choose $h(k_1,k_2) = \frac12\int_{S^1}(k_1^2+k_2^2)dx$, then $h_1 = k_1$ and $h_2 = k_2$ and $H_{-1} = \rr = \begin{pmatrix} k_1&0\\0& k_2\end{pmatrix}$. Calculating (\ref{groupev}) with the natural moving frame we obtain
\[
\begin{pmatrix} I&-u\\ 0&I\end{pmatrix}\begin{pmatrix} 0&u_t\\0&0\end{pmatrix} = \begin{pmatrix} u_1 B^{-1}&0\\ 0&B^{-1}\end{pmatrix} \begin{pmatrix} 0&\rr\\0&0\end{pmatrix}\begin{pmatrix}Bu_1^{-1}&0\\0&B\end{pmatrix}
\]
which results in
\[
u_t = u_1B^{-1}\rr B.
\]
If we now look at the third normalization equations, notice that $i_3$ is either $\begin{pmatrix} k_1&0\\ 0&k_2\end{pmatrix}$ in the real case, or $\begin{pmatrix} k_1&-k_2\\ k_2&k_1\end{pmatrix}$ in the complex case. In either situation, if we choose $\rr = i_3$, then we get $B^{-1}i_3 B = S(u)$, since the third normalization equation in both cases read $BS(u)B^{-1} = i_3$.

The choice of $h$ clearly results in the completely integrable systems the Theorem states.
\end{proof}
Notice that this immediately implies that equation (\ref{Kdv-Sch}) preserves the level set $\k3= \k4=0$, a fact that can also be found directly using the techniques we used in section \ref{3.3}. 

The last results in this paper are proven choosing a different, but also non-local, moving frame. This moving frame corresponds to the Laguerre-Forsyth's canonical form of the Serret-Frenet equations for $\rho$, and it was linked to Grassmannians in \cite{Se}. The invariants generated by this last Grassmannian moving frame evolve following a noncommutative KdV equation. When reduced to these invariants, both brackets (\ref{mainbr}) and (\ref{secbr}) become the biHamiltonian structure for noncommutative KdV.

\begin{theorem} There exists a (non-local) Grassmannian moving frame such that its associated Maurer-Cartan matrix is given by
\[
\begin{pmatrix} 0&I\\ \widehat K&0\end{pmatrix}
\]
where the entries of $\widehat K$ are independent and generating differential invariants. Both structures (\ref{mainbr}) and (\ref{secbr}) (with the choice $L_0 = \begin{pmatrix} 0&0\\ I&0\end{pmatrix}$) reduce to $\K$ as represented by these invariants to produce a biHamiltonian structure for the noncommutative KdV equation. Furthermore, the evolution
\[
u_t = u_1 \widehat K
\]
induces a noncommutative KdV equation on $\widehat K$.
\end{theorem}

\begin{proof}
Since the technique is identical to the previous cases, we will describe the first part of the calculations without too many explanations. If we gauge (\ref{grassK}) or (\ref{grassKc}) by an element of the form 
\[
g = \begin{pmatrix}\Theta&0\\ 0&\Theta\end{pmatrix}
\]
the result is
\[
\begin{pmatrix}\Theta^{-1}\Theta_x+\Theta^{-1}K_0\Theta&I\\ \Theta^{-1} K_1\Theta&\Theta^{-1}\Theta_x+\Theta^{-1}K_0\Theta\end{pmatrix}.
\]
Therefore, choosing $\Theta$ to be the solution of $\Theta_x = - K_0 \Theta$ results on the choice $\widehat K = \Theta^{-1} K_1 \Theta$. Notice that the invariants $k_1$ and $k_2$ are generated by the trace and determinant of $\widehat K$.

We now calculate the reduction of the Poisson brackets. Assume
\[
\var{\Hop} = \begin{pmatrix}A&\h\\ C&B\end{pmatrix}
\]
is the derivative of an appropriate extension of $h$ with $\h = \frac{\delta h}{\delta \kappa}$. Then, equation (\ref{Hcond}) results in the values
\begin{eqnarray*}
A &=&\frac12\h_x-\frac12 D^{-1}(\widehat K\h-\h\widehat K)\\
B &=&-\frac12\h_x-\frac12 D^{-1}(\widehat K\h-\h\widehat K)\\
C &=&-\frac12\h_{xx}+\frac12(\widehat K\h+\h\widehat K).
\end{eqnarray*}
Using this extension, the reduction of (\ref{mainbr}) is given by
\[
\{h, f\}_R(\kappa) = \int \mathrm{tr}\left((C_x+\widehat K A- B \widehat K)\frac{\delta f}{\delta \kappa}\right) dx = \int \h^T \P {\bf f} dx
\]
where 
\[
2\P\h = - \h_{xxx}+(\h \widehat K + \widehat K \h)_x+ \widehat K \h_x+\h_x\widehat K- \widehat KD^{-1}(\widehat K\h-\h\widehat K)+D^{-1}(\widehat K\h-\h\widehat K)\widehat K.
\]
 Choosing $L_0 = \begin{pmatrix} 0&0\\ I&0\end{pmatrix}$ in (\ref{secbr}) and using the expression for the extension $\var{\Hop}$, we have that the reduction of (\ref{secbr}) is given by
\[
\{h, f\}_0(\kappa) = \int \h^T D {\bf f} dx.
\]
These two structure are the ones appearing in \cite{OS} as biHamiltonian structures for noncommutative KdV.
The calculation of the geometric realization is identical to the previous cases.

There is one last point that need to be checked. An arbitrary gauge by a matrix of invariants of a moving frame it does result in a new moving frame, but the entries of its Maurer-Cartan matrix do not necessarily generate all other invariants. Indeed, the proof in \cite{H} shows that if a moving frame is found using normalization constants (as in the first frame we found), then they do indeed generate. But, although perhaps true, a general theorem for any moving frame has not been proved yet. Therefore, we need to check that such is the case here.

We know that $k_1$ and $k_2$ are generated by the entries of $\widehat{K}$. Lengthy but straightforward calculations show that, if $\Theta = \begin{pmatrix} a&b\\ c&d\end{pmatrix}$ and $\widehat{K} = \begin{pmatrix}\widehat{\k1}&\widehat{\k2}\\\widehat{\k3}&\widehat{\k4}\end{pmatrix}$, then 
\[
\frac{a' a}{b b'}=-\frac{\widehat{\k3}}{\widehat{\k2}}, \hskip 3ex \frac{a' b}{ab'}=\frac{k_2^2-k_1^2}{\widehat{\k1} k_1-\widehat{\k4} k_2}+1.
\]
This implies that $n_1 = \frac{a'}{b'}$ and $n_2 = \frac ab$ are functionally generated by $\widehat{\k i}$, $i=1,2,3,4$. Further calculations show that $\k 3\k 4 = -n_1'n_2'$ and 
\[
\frac{\k3'}{\k3} n_2' - 2 \frac{n_2'}{n_1-n_2} = n_2''
\]
concluding that both $\k3$ and $\k 4$ are also formally functionally generated by the entries of $\widehat K$.

\end{proof}

Our final corollaries state the existence of a conformal level set for a decoupled system of KdV equations. It also shows that the signature $(2,2)$ conformal Poisson brackets are equivalent to the bi-Hamiltonian structure for non-commutative KdV $2\times 2$ equations. Although this is immediate under the isomorphism, it was not previously obvious.
 \begin{corollary} There exists a conformally invariant evolution of curves (written in terms of a natural classical moving frame) such that it preserves $\k3 = \k4 = 0$. As $\k3, \k4 \to 0$, the evolution induces a decoupled system of KdV equations on $k_1$ and $k_2$.
 \end{corollary}

 \begin{corollary} Both brackets (\ref{mainbr}) and (\ref{secbr}) reduce to the space of conformal differential invariants to produce a biHamiltonian structure equivalent to that of the non-commutative KdV equation that appears in \cite{OS}.
 \end{corollary}
  The interested reader can find the exact equations for these geometric realizations using the isomorphism with the corresponding Grassmannian evolution.

\end{document}